\documentclass[a4paper,oneside]{amsart}

\usepackage{geometry}
\usepackage[english]{babel}
\usepackage{amsmath}
\usepackage[utf8]{inputenc}
\usepackage[T1]{fontenc}
 
\usepackage{amsthm}
\usepackage{amssymb}
\usepackage[all]{xy}
\usepackage{MnSymbol}
\usepackage[,colorlinks=true,citecolor= DarkBlue,linkcolor=Black]{hyperref}
\usepackage[svgnames,table]{xcolor}
\usepackage{graphicx}
\usepackage{tikz,tkz-tab}

\theoremstyle{definition}
\newtheorem{dfn}{Definition}[section]

\theoremstyle{plain}
\newtheorem{thm}{Theorem}
\newtheorem*{thm*}{Theorem}
\newtheorem{prop}[dfn]{Proposition}
\newtheorem*{prop*}{Proposition}
\newtheorem{lem}[dfn]{Lemma}
\newtheorem*{lem*}{Lemma}

\newtheorem*{sublemma*}{Sub-lemma}
\newtheorem{cor}[dfn]{Corollary}

\newtheorem*{fact*}{Fact}

\theoremstyle{remark}
\newtheorem{rem}[dfn]{Remark}

\newcommand\chech[1]{\mathaccent"7014{#1}}

\newcommand{\R}{\mathbf{R}}

\renewcommand{\sl}{\mathfrak{sl}}

\newcommand{\aff}{\mathfrak{aff}}

\renewcommand{\a}{\mathfrak{a}}
\newcommand{\g}{\mathfrak{g}}
\newcommand{\h}{\mathfrak{h}}    

\newcommand{\n}{\mathfrak{n}}
\newcommand{\so}{\mathfrak{so}}
\newcommand{\p}{\mathfrak{p}}

\newcommand{\e}{\mathrm{e}}
\renewcommand{\d}{\mathrm{d}}

\DeclareMathOperator{\Isom}{Isom}
\DeclareMathOperator{\Ad}{Ad}

\DeclareMathOperator{\Span}{Span}

\DeclareMathOperator{\Kill}{Kill}

\DeclareMathOperator{\GL}{GL}

\DeclareMathOperator{\CO}{CO}

\DeclareMathOperator{\PSL}{PSL}
\DeclareMathOperator{\PO}{PO}

\DeclareMathOperator{\SO}{SO}

\DeclareMathOperator{\Aff}{Aff}

\DeclareMathOperator{\Ima}{Im}

\DeclareMathOperator{\Conf}{Conf}

\DeclareMathOperator{\id}{id}

\DeclareMathOperator{\Mon}{Mon}

\renewcommand{\S}{\mathbf{S}}
\renewcommand{\H}{\mathbf{H}}
\newcommand{\dS}{\mathbf{dS}} 
\newcommand{\Ein}{\mathbf{Ein}}

\renewcommand{\epsilon}{\varepsilon}
\renewcommand{\geq}{\geqslant}
\renewcommand{\leq}{\leqslant}
\renewcommand{\hat}{\widehat}  
\newcommand{\hx}{\hat{x}}
\renewcommand{\tilde}{\widetilde}
\renewcommand{\bar}{\overline}

\title[Essential actions of $\PSL(2,\R)$ on analytic Lorentz manifolds]{Essential conformal actions of $\PSL(2,\R)$ on real-analytic compact Lorentz manifolds}

\author{Vincent Pecastaing}

\begin{document}

\maketitle

\begin{center}
\today
\end{center}

\begin{abstract}
The main result of this paper is the conformal flatness of real-analytic compact Lorentz manifolds of dimension at least $3$ admitting a conformal essential (\textit{i.e.} conformal, but not isometric) action of a Lie group locally isomorphic to $\PSL(2,\R)$. It is established by using a general result of M. Gromov on local isometries of real-analytic $A$-rigid geometric structures. As corollary, we deduce the same conclusion for conformal essential actions of connected semi-simple Lie groups on real-analytic compact Lorentz manifolds. This work is a contribution to the understanding of the Lorentzian version of a question asked by A. Lichnerowicz.
\end{abstract}

\tableofcontents

\section{Introduction}
\label{s:intro}

In this paper, we discuss some geometric aspects of conformal actions of semi-simple Lie groups on compact Lorentz manifolds. The general problem we are interested in is the following: To what extent does the conformal group of a Lorentz manifold $(M,g)$ determine the conformal geometry of $(M,g)$? A beautiful answer to this question was found in Riemannian signature. Motivated by a conjecture of Lichnerowicz, Ferrand and Obata proved that if the conformal group of a Riemannian manifold is strictly larger than its isometry group (even if we change conformally the metric), then this manifold is conformally equivalent to the Möbius sphere or to the Euclidian space of same dimension (\cite{obata71}, \cite{ferrand71}, \cite{ferrand96}).

When we leave the positive definite signature, the situation is more complicated. Since the work of Frances (\cite{frances_ferrand_obata_lorentz}), we know that it is not reasonable to expect a global result as striking as Ferrand-Obata's theorem, even in Lorentzian signature. However, a question remains open about the local conformal geometry of Lorentz manifolds:

\vspace*{0.2cm}

\hspace*{-.5cm} \textit{If the conformal group of a compact Lorentz manifold is not reduced to its isometry group (even after a conformal change of metric), can we conclude that the manifold is conformally flat ?}

\vspace*{0.2cm}

In \cite{alekseevsky}, Alekseevsky gave counter-examples to this question in the non-compact case (see also the work of K\"uhlner and Rademacher in any signature, \cite{kuhnel_rademacher1}, \cite{kuhnel_rademacher2}). Thus, the compactness assumption cannot be removed. Let us also mention that in \cite{frances_lichnerowicz}, Frances answered negatively to this question in $(p,q)$-signature, with $\min(p,q) \geq 2$, and his counterexample has moreover \textit{real-analytic regularity} (one could have expect that with stronger regularity assumption, conformal flatness would have been more plausible). Thus, when we pass from the Riemannian setting to the signature $(p,q)$, with $\min(p,q) \geq 2$, the same problem has drastically different answers, and we would like to know what happens for the intermediate case of Lorentzian metrics.

\vspace*{0.2cm}

Let $M^n$ be a differentiable manifold of dimension $n \geq 3$. Recall that two pseudo-Riemannian metrics $g$ and $g'$ on $M$ are said to be conformal if there exists a positive smooth function $\varphi$ on $M$ such that $g' = \varphi g$. We note $[g] = \{g', \ g' \text{ conformal to } g\}$ the conformal class of $g$. We say that $(M,g)$ is \textit{conformally flat} if every $x \in M$ admits a neighbourhood $U$ such that the restriction $g|_U$ is conformal to a flat metric on $U$. A local diffeomorphism $f : (M,g) \rightarrow (N,h)$ between two pseudo-Riemannian manifolds is said to be conformal if $f^*[h] = [g]$. The conformal group $\Conf(M,g)$ consists of the conformal diffeomorphisms of $(M,g)$. From the rigidity of conformal structures in dimension greater than or equal to $3$, it follows that $\Conf(M,g)$ is a Lie transformation group.

\begin{dfn}
Let $H < \Conf(M,g)$ be a Lie subgroup. We  say that \textit{$H$ acts inessentially on $M$}, or simply $H$ \textit{is inessential}, if there exists $g'$ conformal to $g$ such that $H$ acts on $M$ by isometries of $g'$. If not, we say that $H$ acts \textit{essentially}, or simply that $H$ \textit{is essential}.
\end{dfn}

The present work investigates the situation where a compact Lorentz manifold $(M^n,g)$, $n \geq 3$, admits a conformal essential action of a \textit{semi-simple Lie group}. This hypothesis is of course stronger than what is assumed in the initial question where for instance $\Conf(M,g)$ could only contain an essential one-parameter group or an essential discrete subgroup. 

In \cite{article2}, following a previous investigation of Bader and Nevo, we gave a classification of semi-simple Lie groups without compact factor that can act conformally on a compact Lorentz manifolds of dimension $n \geq 3$. The question we are asking now lies in the continuation of this classification result: once we know what group \textit{can} act, we want to know on which geometry and with which dynamic it actually acts. 

\vspace*{0.2cm}

This geometric problem is well described, in any signature, when the group that acts is semi-simple without compact factor and has \textit{high real-rank}. In fact, the work of Zimmer in \cite{zimmer87} shows that the real-rank of a semi-simple Lie group without compact factor acting by conformal transformations on a compact pseudo-Riemannian manifold of signature $(p,q)$ is bounded by $\min(p,q)+1$. In \cite{bader_nevo} and \cite{frances_zeghib} (and \cite{bader_frances_melnick} in a more general geometric context), the authors investigated the extremal case where the bound is achieved. It turned out that not only such a manifold must be conformally flat, but also that it is a compact quotient of the universal cover of the model space of conformal geometry (the \textit{Einstein universe}, see Section \ref{ss:examples}). Thus, our contribution is mainly the description of conformal actions of non-compact simple Lie groups of real-rank $1$ on compact Lorentz manifolds. Before stating our results, let us describe the situation for \textit{inessential} actions, \textit{i.e.} actions that are isometric after a conformal change of metric.

\vspace*{0.2cm}

First of all, Zimmer proved the following

\begin{thm*}[\cite{zimmer86}]
Let $(M,g)$ be a Lorentz manifold of finite volume and $H$ be a semi-simple simple Lie group without compact factor. If $H$ acts faithfully and isometrically on $(M,g)$, then it is locally isomorphic to $\PSL(2,\R)$.
\end{thm*}

A compact Lorentz manifold admitting an isometric and faithful action of $H \simeq_{\text{loc}} \PSL(2,\R)$ can be easily built. Let $g_K$ be the Killing metric of $H$. It has Lorentz signature and is invariant under left and right translations of $H$ on itself. Choose $\Gamma < H$ any uniform lattice, and set $M = H / \Gamma$. The action of $\Gamma$ on $(H,g_K)$ by right translations being isometric, $g_K$ induces a Lorentz metric $g$ on $M$. Moreover, the action of $H$ on itself by left translations is also isometric for $g_K$ and centralizes the right translations. Therefore, $H$ acts on $M = H / \Gamma$ by isometries of $g$.

Moreover, the geometry of a manifold admitting an inessential action of such Lie groups $H$ was described by Gromov (\cite{gromov}, 5.4.A). He proved that if  $(M,g)$ has finite volume and if $H < \Isom(M,g)$, then $H$ acts \textit{locally freely} everywhere, with Lorentzian orbits. Then, considering the (Riemannian) orthogonal of these orbits, he concluded that some isometric cover of $M$ is a warped product $H \! \! ~_{\omega} \! \! \times N$, where $H$ is endowed with its Killing metric, $N$ is a Riemannian manifold and $\omega : N \rightarrow \R_{>0}$. Moreover, the $H$-action on $M$ can be lifted to the isometric action of $H$ on $H \! \! ~_{\omega} \! \! \times N$ given by left translations on the first factor and trivial on $N$.

\vspace*{0.2cm}

Thus, the situation is well described for inessential action. We now come to our main result.

\begin{thm}
\label{thm:main}
Let $(M^n,g)$, $n \geq 3$, be a real-analytic compact connected Lorentz manifold admitting a faithful conformal action of a connected Lie group locally isomorphic to $\PSL(2,\R)$. Then,
\begin{itemize}
\item either this action is inessential, \textit{i.e.} $\exists g_0 \in [g]$ such that $H$ acts by isometries of $g_0$ ;
\item or $(M,g)$ is conformally flat.
\end{itemize}
\end{thm}

At first sight, this result could appear restrictive since it is formulated for a precise Lie group action. However, (local) copies of $\PSL(2,\R)$ sit in every non-compact simple Lie group. So, a conformal essential action of a semi-simple Lie group without compact factor always yields conformal actions of Lie groups locally isomorphic to $\PSL(2,\R)$. The question will only be to determine if these Lie subgroups are essential, what will not be a difficult problem.

\begin{cor}
\label{cor:ss_conformally_flat}
Let $(M^n,g)$, $n \geq 3$, be a real-analytic compact connected Lorentz manifold. If a connected semi-simple Lie group acts faithfully, conformally and essentially on $M$, then $(M,g)$ is conformally flat.
\end{cor}

Remark that any conformal action of a compact Lie group $H$ on a pseudo-Riemannian manifold $(M,g)$ is inessential. Indeed, $H$ admits a finite Haar measure $\mu$ and the metric $g' := \int_H h^* \! g \, \d \mu (h)$ is $H$-invariant. Thus, Corollary \ref{cor:ss_conformally_flat} has an interest for non-compact semi-simple Lie groups. In this case, essentiality is characterized by the following property, valid in smooth regularity.

\begin{prop}
\label{prop:semi_simple_inessential}
Let $H$ be a non-compact, connected, semi-simple Lie group acting faithfully and conformally on a compact connected Lorentz manifold $(M^n,g)$, $n \geq 3$. Let $\h_1 \subset \h$ be the sum of the non-compact simple ideals of $\h$ and let $H_1 < H$ be the corresponding connected Lie subgroup. Then, $H$ is inessential if and only if $\h_1 \simeq \sl(2,\R)$ and $H_1$ acts locally freely on $M$.
\end{prop}

Remark that Proposition \ref{prop:semi_simple_inessential} implies that a conformal action of a Lie group locally isomorphic to $\PSL(2,\R)$ is inessential if and only if it is everywhere locally free.

\vspace*{0.2cm}

The question of conformal flatness for \textit{smooth} compact Lorentz manifolds admitting an essential action of a Lie group locally isomorphic to $\PSL(2,\R)$ is still open. Here, we take advantage of the analytic regularity to present a relatively short and efficient proof. Using a general theory of Gromov on real-analytic $A$-rigid geometric structures, we exhibit a local conformal flow, which \textit{does not come from the action} of the Lie group, and we deduce the conformal flatness from its dynamic. The existence of this flow can be compared to the statement of Gromov's celebrated centralizer theorem (\cite{gromov}, see also \cite{feres02} Cor. 7.5, \cite{morris_zimmer} 5.A).

If we suspect that Theorem \ref{thm:main} is in fact a smooth result, the proof that is given here seems to be a ``real-analytic proof'', in the sense that one encounters fundamental obstacles, which are inherent to Gromov's theory, when trying to generalize it to smooth Lorentz manifolds.

\subsection*{Organization of the proof}

The proof of Theorem \ref{thm:main} is the main thread of this article, and is structured as follows. We start by proving that inessential actions of Lie groups locally isomorphic to $\PSL(2,\R)$ are characterized by the fact that they are everywhere locally free, and at the same time we prove Proposition \ref{prop:semi_simple_inessential} and reduce the proof of Corollary \ref{cor:ss_conformally_flat} to the one of Theorem \ref{thm:main}. 

Once we know that essential actions are characterized by the existence of orbits with \textit{small dimension}, we prove in Section \ref{s:fixed_points} that, moreover, there always exists a closed invariant subset in which every orbit has dimension $1$ or $2$. Until that time, no analyticity assumption is necessary.

\vspace*{0.2cm}

When the manifold contains a $1$-dimensional orbit, anterior results will quickly give the conformal flatness (it is done in Section \ref{ss:orbit_dim1}). Thus, we will assume that there exists a closed invariant subset that exclusively contains orbits of dimension $2$. 

It is at this moment that the analyticity hypothesis becomes crucial. Conformal structures in dimension at least $3$ are rigid in Gromov's sense (\cite{gromov}), so that its strong theory on the behaviour of local isometries applies in our context. The main step in its theory is an integrability result he called ``Frobenius theorem''. Without going into details, the question is to see when an ``isometric $r$-jet'' between two points $x$ and $y$, which can be thought as an infinitesimal map only defined at $x$, can be extended into a local isometry between neighbourhoods of $x$ and $y$. For compact real-analytic rigid structures, Gromov proved that isometric $r$-jets always give rise to local isometries, but for smooth structures, the result is true only in an \textit{open and dense subset} of the manifold (\cite{gromov} 3.3, 3.4, \cite{benoist} p.14).

In our situation, we will use this integrability result at a precise point $x_0$ in the invariant closed subset we have exhibited. We will not use its original version but a more recent formulation due to Melnick in the setting of \textit{Cartan geometries} (\cite{melnick}, \cite{article1}). It will provide a local conformal vector field $A^*$ defined near $x_0$, thanks to the analyticity of the Lorentz manifold. If the manifold was only assumed to be smooth, the point $x_0$ we are interested in could unfortunately be outside the above mentioned open dense subset and we could not apply this Frobenius theorem. The existence and the description of this conformal vector field are postponed to Section \ref{s:proof_local_flow}.

Admitting temporarily the existence of $A^*$, we will perform the proof of Theorem \ref{thm:main} in Section \ref{s:conformal_flatness}. We will establish that the Weyl-Cotton tensor vanishes on a neighbourhood of $x_0$ by considering the dynamic of the flow $\phi_{A^*}^t$. The conformal flatness of the whole manifold will directly follow from its connectedness and analyticity.

\subsection*{Conventions and notations}

In this paper, ``manifold'' means a differentiable manifold. By default, the regularity is assumed to be smooth and we will precise when the analyticity is required. As usual, we will use the fraktur font to denote the Lie algebra of a Lie group. If $M$ is a manifold, we note $\mathfrak{X}(M)$ the Lie algebra of vector fields defined on $M$. We note $\Kill(M,[g])$ the Lie algebra of \textit{conformal (Killing) vector fields} of $M$, \textit{i.e.} infinitesimal generators of conformal diffeomorphisms. If $\dim M \geq 3$, then $\Kill(M,[g])$ is finite dimensional. We call \textit{$\sl(2)$-triple} of a Lie algebra any non-zero triple $(X,Y,Z)$ in this Lie algebra satisfying the relations $[X,Y] = Y$, $[X,Z] = -Z$ and $[Y,Z] = X$.

\vspace*{0.2cm}

\hspace*{-0.55cm} \textbf{Fixed global notations.} In all this paper, $M^n$ is a compact connected manifold of dimension $n \geq 3$, endowed with a Lorentz metric $g$, and $H$ is a connected Lie group that acts faithfully conformally on $(M,g)$. The $H$-action gives rise to an \textit{infinitesimal action} of $\h$, \textit{i.e.} an embedding $\h \hookrightarrow \Kill(M,[g])$ into the Lie algebra of conformal vector fields given by $X \mapsto (\frac{\d}{\d t}|_{t=0} \, \e^{-tX} \! \! .x)_{x\in M}$. We will identify $\h$ with its image. For all $x \in M$, the notation $\h_x$ refers to the Lie algebra of the stabilizer of $x$, $\h_x = \{X \in \h, \ X_x = 0\}$.

\vspace*{0.4cm}

\hspace*{-0.55cm} \textbf{Acknowledgement.} \textit{This work has been done during my PhD and I would like to deeply thank my advisor, Charles Frances, for his constant support.}

\section{Inessential actions of $\PSL(2,\R)$}
\label{s:inessential}

The theorem of Zimmer cited in the introduction shows that if a semi-simple Lie group without compact factor acts conformally inessentially on a compact Lorentz manifold, then it is locally isomorphic to $\PSL(2,\R)$. We also recalled a result of Gromov that moreover shows that its action is everywhere locally free. The main purpose of this section is to establish that conversely, locally free conformal actions of such groups on compact Lorentz manifolds are inessential.

\begin{prop}
\label{prop:sl2_inessential}
Let $(M^n,g)$, $n \geq 3$, be a compact connected Lorentz manifold that admits a faithful conformal action of a connected Lie group $H$ locally isomorphic to $\PSL(2,\R)$. If this action is everywhere locally free, then it is inessential, \textit{i.e.} there exists $g'$ conformal to $g$ such that $H < \Isom(M,g')$.
\end{prop}

The proof is based on a geometric property, derived from a generalization of Zimmer's embedding theorem, due to Bader, Frances and Melnick (\cite{bader_frances_melnick}). This property is valid for non-locally free actions, and will be reused later. We have isolated it in the following section.

\subsection{A geometric property of general conformal actions}

Let $H$ be a connected Lie group locally isomorphic to $\PSL(2,\R)$, acting conformally on the compact Lorentz manifold $(M,g)$. Let $\mathcal{O}_x = H.x$ be the $H$-orbit of a point $x \in M$. Identifying $\h / \h_x \simeq T_x \mathcal{O}_x$, let $q_x$ be the quadratic form on $\h / \h_x$ obtained by restricting the ambiant metric $g_x$ to $T_x \mathcal{O}_x$.

\vspace*{0.2cm}

\begin{prop}
\label{prop:zimmer_conformal}
Let $S < H$ be a proper, closed, connected subgroup such that $\Ad_{\h}(S) \subset \Ad_{\h}(H) \simeq \PSL(2,\R)$ is not an elliptic one-parameter subgroup. Then, every $S$-invariant closed subset of $M$ contains a point $x$ such that $\Ad_{\h}(S)\h_x \subset \h_x$ and such that the induced action of $\Ad_{\h}(S)$ on $\h / \h_x$ is conformal with respect to $q_x$.
\end{prop}

To the Lorentzian conformal structure $(M, [g])$ corresponds what we call a \textit{normalized Cartan geometry} modeled on the Lorentzian Einstein universe. It is a geometric structure on $M$ such that any conformal object on $M$ translates into its framework, and in which problems of conformal geometry can be simpler. Although there is an intrinsic general theory of Cartan geometries, we have chosen to give a minimal presentation and we refer the reader to \cite{kobayashi}, \cite{sharpe} and \cite{cap_slovak} for a more complete approach. The proof of Proposition \ref{prop:zimmer_conformal} starts after this presentation.

\subsubsection{The associated Cartan geometry} 
\label{ss:cartan_geometry}
The definition of the Einstein universe of Lorentzian signature is recalled in Section \ref{ss:examples}. It is a  compact, conformally homogeneous, Lorentz manifold $\Ein^{1,n-1}$, with conformal group $G := PO(2,n)$. Thus, $G$ acts transitively on it so that, if $P$ denotes the isotropy at a given point, $\Ein^{1,n-1} \simeq G/P$ as $G$-homogeneous spaces. The idea is now to transpose the infinitesimal properties of this homogeneous space on an arbitrary manifold. In what follows, $G$ still denote $\PO(2,n)$ and $P < G$ the stabilizer of a (fixed) given point.

\vspace*{0.2cm}

The data of a conformal class $[g]$ of Lorentz metrics on $M$ canonically defines a geometric structure on $M$, called the associated normalized \textit{Cartan geometry} modeled on the Lorentzian Einstein's universe, which is the data of a $P$-principal fiber bundle $\pi : \hat{M} \rightarrow M$ and of a $\g$-valued $1$-form $\omega \in \Omega^1(\hat{M},\g)$ satisfying the following properties:
\begin{enumerate}
\item For all $\hx \in \hat{M}$, $\omega_{\hx} : T_{\hx}\hat{M} \rightarrow \g$ is a linear isomorphism ;
\item For all $X \in \p$, $\omega(X^*) = X$, where $X^*$ denotes the fundamental vector field associated to the right action of $\exp(tX)$ ;
\item For all $p \in P$, $(R_p)^*\omega = \Ad(p^{-1})\omega$,
\end{enumerate}
with an additional normalization assumption on $\omega$ which we do not need to detail here. The fiber bundle is called the \textit{Cartan bundle} and the $1$-form is called the \textit{Cartan connection}. This correspondence is such that if $(M,[g])$ and $(N,[h])$ are two conformal Lorentzian structures, then a diffeomorphism $f : M \rightarrow N$ is conformal if and only if there exists a bundle morphism $\hat{f} : \hat{M} \rightarrow \hat{N}$, with base map $f$ and such that ${\hat{f}}^* \omega_N = \omega_M$. Moreover, $\hat{f}$ is completely determined by $f$ (see \cite{cap_slovak}, Prop. 1.5.3), so that considering a conformal map $f$ is the same than considering a bundle morphism $F : \hat{M} \rightarrow \hat{N}$ such that $F^* \omega_N = \omega_M$.

For any $x \in M$ and $\hx$ over $x$, we have a linear isomorphism $\varphi_{\hx} : T_xM \rightarrow \g / \p$ defined by $\varphi_{\hx}(v) = \omega_{\hx}(\hat{v})$ mod.$\p$, for all $v \in T_xM$ and $\hat{v} \in T_{\hx}\hat{M}$ such that $\pi_* \hat{v} = v$ (it is well defined by property (2)). There exists a conformally $\Ad_{\g}(P)$-invariant Lorentzian quadratic form $Q$ on $\g / \p$ such that, by construction of the Cartan geometry $(M,\hat{M},\omega)$,  $\varphi_{\hx}$ sends any $g_x$ in the conformal class on a positive multiple of $Q$.

If $f \in \Conf(M,[g])$, then the condition $\hat{f}^* \omega = \omega$ together with property (1) stated above says that $\hat{f}$ preserves a global framing on $\hat{M}$. This observation implies that the action of $\Conf(M,[g])$ on $\hat{M}$ is free and proper. In particular, if $M$ is connected and if $\hx \in \hat{M}$ is given, then $\hat{f}$ is completely determined by the data of $\hat{f}(\hx)$. At the infinitesimal level, the fact that a vector field $X \in \mathfrak{X}(M)$ is conformal translates into the existence of a vector field $\hat{X} \in \mathfrak{X}(\hat{M})$ called the lift of $X$ such that $\hat{X}$ commutes to the right $P$-action on $\hat{M}$, $\pi_* \hat{X} = X$ and $\mathcal{L}_{\hat{X}} \omega = 0$. A conformal vector field $X$ is determined by the evaluation of its lift $\hat{X}$ at any given point $\hx$.

\begin{proof}[Proof of Proposition \ref{prop:zimmer_conformal}]
As we mentioned earlier, this Proposition is based on a Theorem which is the main result of \cite{bader_frances_melnick}. It is formulated in the setting of general Cartan geometries and general Lie group actions. We have adapted the proof to our context.

Let $\pi : \hat{M} \rightarrow M$ and $\omega$ be the Cartan bundle and the Cartan connection defined by $(M,[g])$. Let $\iota : \hx \in \hat{M} \rightarrow \iota_{\hx} \in \Mon(\h,\g)$ be the map defined by setting $\forall	X \in \h$, $\iota_{\hx}(X) = \omega_{\hx}(\hat{X}_{\hx})$, where $\Mon(\h,\g)$ denotes the variety of injective linear maps $\h \rightarrow \g$. Let $H \times P$ act on $\hat{M}$ via $h.\hx.p^{-1}$ and on $\Mon(\h,\g)$ via $\Ad_{\g}(p) \circ \alpha \circ \Ad_{\h}(h^{-1})$. We verify that $\iota$ is $(H\times P)$-equivariant. Define $W := \bar{P} \backslash \Mon(\h,\g)$, where $\bar{P}$ denotes the Zariski closure of $\Ad_{\g}(P)$ in $\GL(\g)$. This quotient $W$ is a stratified variety endowed with a stratified algebraic action of $\GL(\h)$ (by a result of Rosenlicht, see \cite{gromov}, 2.2). Moreover, $\iota$ induces an $H$-equivariant map
\begin{equation*}
\psi : M \rightarrow W,
\end{equation*}
when $H$ acts on $W$ via $\Ad_{\h}(H) \subset \GL(\h)$.

By hypothesis, $S$ is either a one-parameter subgroup of $H$, or is conjugated to the identity component of the affine group $\Aff^+(\R)$. Thus, it is amenable, so that for every non-empty closed $S$-invariant subset $F \subset M$, there exists a finite $S$-invariant measure $\mu$ on $M$ which is supported on $F$. Since $\psi$ is equivariant, $S$ acts on $W$ by preserving $\psi_* \mu$. In fact, we have more: the Zariski closure $\bar{S}$ of $\Ad_{\h}(S)$ in $\GL(\h)$ acts on $W$ and preserves $\psi_* \mu$ (\cite{zimmer_ergodic}, 3.2.4). By hypothesis, $\bar{S}$ is either an hyperbolic or unipotent one-parameter subgroup of $\GL(\h)$ or is isomorphic to the affine group $\R^* \ltimes \R$. In particular, $\bar{S}$ does not admit non-trivial algebraic cocompact subgroups. Using a generalization of Borel's density theorem (\cite{furstenberg}, Lemma 3), we conclude that $\psi_* \mu$ almost every point in $W$ is fixed by $\bar{S}$.

In particular, there exists $x \in F$ such that $S.\psi(x) = \{\psi(x)\}$. This means that for any $\hx \in \pi^{-1}(x)$, we have $S.\iota_{\hx} \subset P.\iota_{\hx}$, or equivalently that for any $s \in S$, there is $p \in P$ such that $\iota_{\hx}$ conjugates the adjoint action of $s$ on $\h$ to the adjoint action of $p$ on $\iota_{\hx}(\h)$. Now, any $X \in \h$ vanishes at $x$ if and only if for some (equ. for all) $\hx \in \pi^{-1}(x)$, the lift $\hat{X}_{\hx}$ is vertical, or equivalently $\iota_{\hx}(X) \in \p$. Therefore, $X \in \h_x \Rightarrow \iota_{\hx}(X) \in \p \Rightarrow \Ad_{\g}(p) \iota_{\hx}(X) = \iota_{\hx}(\Ad_{\h}(s)X) \in \p \Rightarrow \Ad_{\h}(s)X \in \h_x$.

This proves that $\Ad_{\h}(S)$ preserves the stabilizer $\h_x$ and induces a linear action of $S$ on the quotient $\h / \h_x$, that we note $\bar{\Ad}$. Since $\iota_{\hx}(\h_x) \subset \p$, the map $\iota_{\hx}$ induces a linear map $\psi_{\hx} : \h / \h_x \rightarrow \g / \p$ and it comes from the definitions that $\psi_{\hx}$ coincides with the restriction of $\varphi_{\hx}$ to $T_x \mathcal{O}_x \simeq \h / \h_x$ (this map is defined in Section \ref{ss:cartan_geometry}). Thus, $\psi_{\hx} : (\h/\h_x , q_x) \rightarrow (\g / \p, Q)$ is a conformal linear injective map. The property $S. \iota_{\hx} \subset P. \iota_{\hx}$ implies $\psi_{\hx}(\bar{\Ad}(s)\bar{X}) = \bar{\Ad}(p) \psi_{\hx}(\bar{X})$, where the bars mean that we are in the quotient $\h / \h_x$ or $\g / \p$. Finally, we compute
\begin{align*}
\forall \bar{X} \in \h / \h_x, \ q_x(\bar{\Ad}(s) \bar{X}) & = \lambda Q(\psi_{\hx}(\bar{\Ad}(s) \bar{X})) \text{ for some } \lambda > 0 \text{ since } \psi_{\hx} \text{ is conformal} \\
& = \lambda Q (\bar{\Ad}(p) \psi_{\hx}(\bar{X})) \\
& = \lambda \lambda' Q(\psi_{\hx}(\bar{X})) \text{ for some } \lambda' > 0 \text{ since } \bar{\Ad}(p) \in \Conf(\g/\p , Q) \\
& = \lambda' q_x(\bar{X}),
\end{align*}
proving that $\bar{\Ad}(s) \in \Conf(\h / \h_x, q_x)$.
\end{proof}

\subsection{Locally free conformal actions are inessential: Proof of Proposition \ref{prop:sl2_inessential} }

We assume now that $H \simeq_{\text{loc}} \PSL(2,\R)$ acts conformally and locally freely on $(M,g)$, \textit{i.e.} $\h_x = 0$ for all $x \in M$. The Lie algebra $\h$ being isomorphic to $\sl(2,\R)$, we have an $\sl(2)$-triple $(X,Y,Z)$ of conformal vector fields of $(M,g)$, and our assumption is that they are everywhere linearly independent. Let $S:=\{\e^{tX}\}_{t \in \R} < H$.

\begin{lem}
\label{lem:signature_zimmer_point}
Every closed, $S$-invariant subset of $M$ contains a point $x$ such that the orbit $H.x$ has Lorentz signature. Precisely, $X_x$ is space-like, $Y_x$ and $Z_x$ are light-like and orthogonal to $X_x$.
\end{lem}

Before proving this lemma, we give the following definition.

\begin{dfn}
Let $V$ be a finite-dimensional vector space, endowed with a Lorentz quadratic form $q$. If $V' \subset V$ is a subspace, then the restriction $q|_{V'}$ is said to be \textit{sub-Lorentzian}. 
\end{dfn}

A sub-Lorentzian quadratic form is either non-degenerate, with Riemannian or Lorentzian signature, or degenerate and non-negative, with a $1$-dimensional kernel. In any event, a subspace that is totally isotropic with respect to a sub-Lorentzian quadratic form has dimension at most $1$.

\begin{proof}[Proof (Lemma \ref{lem:signature_zimmer_point})]
By Proposition \ref{prop:zimmer_conformal} and since $\h_x = 0$ for all $x \in M$, any closed $S$-invariant subset contains a point $x$ such that $\Ad_{\h}(S) \subset \Conf(\h,q_x)$. We note $a_t := \Ad_{\h}(\e^{tX})$. Since $a_t$ is linear and conformal with respect to $q_x$, there exists $\lambda \in \R$ such that $a_t^* q_x = \e^{\lambda t} q_x$. 

We claim that at $x$, the vector fields $Y$ and $Z$ are light-like and that $\lambda = 0$. To see it, assume for instance that $q_x(Y) \neq 0$. Since $q_x(a_t Y) = \e^{2t}q_x(Y) = \e^{\lambda t}q_x(Y)$, we would have $\lambda = 2$. But on the other hand, $q_x(a_t X)) = q_x(X)$, $q_x(a_t Z) = \e^{-2t}q_x(Z)$ and if we note $b_x$ the bilinear form on $\h$ associated to $q_x$, we also have $b_x(a_t X,a_t Z) = \e^{-t} b_x(X,Z)$. Thus, from $a_t^* q_x = \e^{2t}q_x$ we would deduce that the plane $\Span(X_x,Z_x)$ is totally isotropic with respect to $q_x$, contradicting that $q_x$ is sub-Lorentzian. Therefore, $Y_x$ is light-like and a similar reasoning also gives $q_x(Z)=0$. Furthermore, $Y_x$ and $Z_x$ cannot be orthogonal, and since $b_x(a_t Y, a_t Z) = b_x(Y,Z) = \e^{\lambda t}b_x(Y,Z)$, we obtain $\lambda = 0$, meaning that $\{a_t\}_{t \in \R}$ is in fact a flow of linear isometries of $(\h,q_x)$. Because $b_x(a_t X, a_t Y)=\e^t b_x(X,Y)$ and $b_x(a_t X, a_t Z) = \e^{-t} b_x(X,Z)$, both $Y_x$ and $Z_x$ are orthogonal to $X_x$. The plane $\Span(Y_x,Z_x)$ being Lorentzian, $X_x$ must be space-like.
\end{proof}

Now, consider the closed subset $F = \{x \in M \ | \ g_x(X,X) \leq 0\}$. It is $S$-invariant since $\phi_X^t$ is conformal with respect to $g$. By Lemma \ref{lem:signature_zimmer_point}, it must be empty, \textit{i.e.} $X$ \textit{is every space-like}. This fact is very useful as the following lemma shows (remark that it is valid in any signature).

\begin{lem}[\cite{obata70}, Theorem 2.4]
\label{lem:centralizer_inessential}
Let $X' \in \Kill(M,[g])$ be a conformal vector field. If $X'$ is nowhere light-like, then $\forall f \in \Conf(M,g)$ such that $f^* X' = X'$, we have $f \in \Isom(M, \frac{g}{|g(X',X')|})$.
\end{lem}

\begin{proof}
Let $\varphi \in \mathcal{C}^{\infty}(M)$ such that $f^* g = \varphi g$. If $\psi := g(X',X')$, we have $\psi(f(x)) = [f^*g]_x(X',X') = \varphi(x) \psi(x)$ since $f_* X_x' = X_{f(x)}'$. Therefore, if we note $g_0 := g / g(X',X')$,
\begin{equation*}
[f^* g_0]_x = \frac{[f^* g]_x}{\psi(f(x))} = \frac{\varphi(x)g_x}{\varphi(x)\psi(x)} = [g_0]_x,
\end{equation*}
proving that $f \in \Isom(M,g_0)$.
\end{proof}

In our case, if we rescale conformally the metric $g$ by $g(X,X) > 0$, the flow $\phi_X^t$ becomes isometric. We still note $g$ this new metric. We now prove that $\{\phi_Y^t\}$ and $\{\phi_Z^t\}$ are also isometric with respect to $g$. Since $[X,Y] = Y$, we have $(\phi_X^t)_* Y_x = \e^{-t} Y_{\phi_X^t(x)}$, and because $\{\phi_X^t\} \subset \Isom(M,g)$, we obtain $g_{\phi_X^t(x)}(Y,Y) = \e^{2t} g_x(Y,Y)$ and $g_{\phi_X^t(x)}(X,Y) = \e^t g_x(X,Y)$. By compactness of $M$, the maps $\{x \mapsto g_x(Y,Y)\}$ and $\{x \mapsto g_x(X,Y)\}$ are bounded. Thus, the previous relations imply that $Y$ is everywhere isotropic and orthogonal to $X$. Now, the relation $[Y,X] = -Y$ gives $(\phi_Y^t)_* X_x = X_{\phi_Y^t(x)} + t Y_{\phi_Y^t(x)}$. Let $\lambda(x,t) > 0$ be such that $[(\phi_Y^t)^* g]_x = \lambda(x,t) g_x$. Using the previous relations, we get
\begin{equation*}
\lambda(x,t) g_x(X,X) = g_{\phi_Y^t(x)}(X,X).
\end{equation*}
By construction, the map $\{x \mapsto	g_x(X,X)\}$ is constant (equal to $1$). This gives us $\lambda(x,t) \equiv 1$, \textit{i.e.} $\phi_Y^t$ is an isometry of $g$. A strictly similar reasoning shows that the flow of $Z$ is also isometric with respect to $g$, and finally we get $H < \Isom(M,g)$ by connectedness.

\subsection{General essential actions of semi-simple Lie groups}

Let $H$ be a connected semi-simple Lie group acting conformally on the compact Lorentz manifold $(M,g)$. We regroup the simple ideals of $\h$ so that $\h = \h_1 \oplus \h_2$ with $\h_1$ semi-simple without compact factor, or trivial, and $\h_2$ compact and semi-simple, or trivial. Let $H_1 < H$ be the connected Lie subgroup corresponding to $\h_1$. We distinguish four cases:
\begin{enumerate}
\item \label{case1} $H_1$ is not locally isomorphic to $\PSL(2,\R)$ and $H_1 \neq \{\id\}$ ;
\item \label{case2} $H_1$ is locally isomorphic to $\PSL(2,\R)$ and essential ;
\item \label{case3} $H_1$ is locally isomorphic to $\PSL(2,\R)$ and inessential ;
\item \label{case4} $H_1 = \{\id\}$.
\end{enumerate}
We claim that $H$ is essential if and only if we are in case (\ref{case1}) or (\ref{case2}). Indeed, if we are in case (\ref{case1}) or (\ref{case2}), then $H_1$ is essential, implying that $H$ is also essential. If we are in case (\ref{case3}), let $(X,Y,Z)$ be an $\sl(2)$-triple in $\h_1$. We know that $X$ is everywhere space-like and that if we rescale $g$ by $g(X,X)$, then $H_1$ acts by isometries of $g$ and any conformal diffeomorphism that commutes with $X$ is isometric with respect to $g$ (proof of Proposition \ref{prop:sl2_inessential} and Lemma \ref{lem:centralizer_inessential}). Since $\h_1$ and $\h_2$ commute, the flow of any conformal vector field of $\h_2$ centralizes $X$, so it preserves the metric $g$. Thus, $H < \Isom(M,g)$ by connectedness, proving that $H$ is inessential. If we are in case (\ref{case4}), then $H$ is a compact Lie group and must be inessential.
 
\vspace*{0.2cm} 
 
This proves Proposition \ref{prop:semi_simple_inessential}. We finish this section by proving Theorem \ref{thm:main} $\Rightarrow$ Corollary \ref{cor:ss_conformally_flat}.

\vspace*{0.2cm}

Assume that Theorem \ref{thm:main} is true and that $H$ is essential and $(M,g)$ real-analytic. If we are in case (\ref{case2}), then Theorem \ref{thm:main} directly gives the conformal flatness of $(M,g)$. If we are in case (\ref{case1}), then the following lemma yields connected Lie subgroups of $H_1$ locally isomorphic to $\PSL(2,\R)$ that act essentially, and we are done.

\begin{lem}
\label{lem:semi_simple_essential}
Let $H_1$ be a connected semi-simple Lie group without compact factor. Assume that $\h_1 \neq \sl(2,\R)$ and that $H_1$ acts conformally on a compact Lorentz manifold $(M,g)$. Then, every Lie subgroup of $H_1$ locally isomorphic to $\PSL(2,\R)$ acts essentially on $(M,g)$.
\end{lem}

\begin{proof}
We note $\Kill(M,g)$ the Lie algebra of Killing vector fields of $g$, \textit{i.e.} the infinitesimal generators of isometries of $g$. Let $\{X,Y,Z\} \subset \h_1$ be an $\sl(2)$-triple, to which corresponds a connected Lie subgroup $H' < H_1$ locally isomorphic to $\PSL(2,\R)$. Assume that $H'$ is inessential. We know that it is the same than assuming that $X,Y,Z$ are everywhere linearly independent, and the proof of Proposition \ref{prop:sl2_inessential} shows that $X$ is everywhere space-like. Since $\{X,Y,Z\}$ is an $\sl(2)$-triple in $\h_1$, $X$ must be in some Cartan subspace $\a \subset \h_1$. The end of the proof is now similar to the one of Proposition \ref{prop:sl2_inessential}.

Let $\Delta \subset \a^*$ be the restricted root system associated to $\a$ and $\h_1 = \h_0 \oplus \bigoplus_{\alpha \in \Phi} \h_{\alpha}$ be the corresponding restricted root-space decomposition of $\h_1$. Rescale the metric by $g(X,X)$. By Lemma \ref{lem:centralizer_inessential}, we have $\h_0 \subset \Kill(M,g)$. Let $\alpha \in \Delta$. If $\alpha(X) = 0$, then $X$ is centralized by $\h_{\alpha}$ and we also have $\h_{\alpha} \subset \Kill(M,g)$. If $\alpha(X) \neq 0$, then let $Y_{\alpha}$ be a non-zero element of $\h_{\alpha}$. If we replace $X$ by $X / \alpha(X)$, then $X \in \Kill(M,g)$, $g(X,X) \equiv$ constant and $[X,Y_{\alpha}] = Y_{\alpha}$. The end of the proof of Proposition \ref{prop:sl2_inessential} can be applied to the couple $(X,Y_{\alpha})$ and we obtain $Y_{\alpha} \in \Kill(M,g)$. Finally, all of $H_1$ acts by isometries of $g$ and $H_1$ must be locally isomorphic to $\PSL(2,\R)$.
\end{proof}

\section{Orbits near fixed points}
\label{s:fixed_points}

Our question on essential conformal actions of semi-simple Lie groups is henceforth reduced to the case where the group is locally isomorphic to $\PSL(2,\R)$. Thus, until the end of the article, $H$ is assumed to be locally isomorphic to $\PSL(2,\R)$ and we assume that its action on $(M,g)$ is essential. By Proposition \ref{prop:sl2_inessential}, we know that the closed $H$-invariant subset $F_{\leq 2} := \{x \in M \ | \ \dim(H.x) \leq 2\}$ is non-empty. The proof of Theorem \ref{thm:main} will take place in this subset: using dynamical methods, we will prove that an $H$-orbit in $F_{\leq 2}$ admits a conformally flat neighbourhood. Prior to this, we establish in this section the following dynamical property.

\begin{prop}
\label{prop:no_fixed_point}
There exists a closed $H$-invariant subset $F \subset F_{\leq 2}$ which does not contain global fixed points of the $H$-action.
\end{prop}

Of course, this proposition has an interest when there exists a global fixed point of the action, \textit{i.e.} a point $x \in M$ such that $H.x = \{x\}$, what we will assume in this section. Before starting the proof of this proposition, we introduce two examples of essential conformal actions.

\subsection{Examples of essential actions}
\label{ss:examples}

If $k \geq 2$, let $\R^{1,k}$ be the $(k+1)$-dimensional Minkowski space and $\Gamma = <2 \id>$ be the (conformal) group generated by a non-trivial homothety. Naturally, $\Gamma$ acts properly discontinuously on $\R^{1,k} \setminus \{0\}$ and is centralized by the linear action of $\SO(1,k)$. Therefore, $\SO(1,k)$ acts conformally on the quotient $(\R^{1,k} \setminus \{0\}) / \Gamma$, usually called a \textit{Hopf manifold}. It is a conformally flat Lorentz manifold, diffeomorphic to $\S^1 \times \S^{k-1}$. If we decompose orthogonally $\R^{1,k} = \R^{1,2} \oplus \R^{0,k-2}$ and let $\PSL(2,\R) \simeq \SO_0(1,2)$ act on $\R^{1,k}$ via $(g,(x,y)) \mapsto (gx,y)$, we obtain a conformal action of $\PSL(2,\R)$ on this compact Lorentz manifold. All its orbits are $2$-dimensional or reduced to a fixed point. In particular, this action is nowhere locally free: it is a first example of essential action of $\PSL(2,\R)$. We can observe the conclusion of Proposition \ref{prop:no_fixed_point}: Let $F$ be the projection of $\R^{1,2} \setminus \{0\}$ in the Hopf manifold. Then, $F$ is closed, $\PSL(2,\R)$-invariant and every orbit in $F$ is $2$-dimensional.

\vspace*{0.2cm}

Let $n \geq 3$. We define (as a manifold) the \textit{Einstein universe} of Lorentzian signature as being the projection of $\mathcal{L}^{2,n} \setminus \{0\}$ in $\R P^{n+1}$, where $\mathcal{L}^{2,n}$ denotes the light-cone of $\R^{2,n}$. Thus, $\Ein^{1,n-1}$ is an $n$-dimensional smooth compact projective variety, on which the group $\PO(2,n)$ acts transitively. It naturally inherits a conformal class $[g]$ of Lorentz metrics from the ambient quadratic form of $\R^{2,n}$, and $[g]$ is invariant under the action of $\PO(2,n)$. If we decompose orthogonally $\R^{2,n} = \R^{1,2} \oplus \R^{1,n-2}$ and let $O(1,2)$ act on $\R^{2,n}$ via $(g,(x,y)) \mapsto (gx,y)$, then we obtain a conformal action of $\PO(1,2)$ on $\Ein^{1,n-1}$ that is also essential (same argument). If $F$ denotes the projection of $(\mathcal{L}^{1,2} \times \{0\}) \setminus \{0\}$ in $\R P^{n+1}$, where $\mathcal{L}^{1,2}$ is the light-cone of $\R^{1,2}$, then $F$ is diffeomorphic to a circle on which $\PO(1,2)$ acts transitively.

\vspace*{0.3cm}

Coming back to the general situation, since $H$ is simple, we can describe locally its orbits near any fixed point thanks to the following \textit{local linearization} property of conformal actions. We will then see that some $2$-dimensional orbits cannot accumulate in the neighbourhood of any fixed point, and it will be enough to consider the closure of such orbits.

\begin{prop}[\cite{frances_zeghib}, Theorem 7]
\label{prop:linearizability}
Let $H'$ be a simple Lie group. If  $H'$ acts conformally on a pseudo-Riemannian manifold $(M',g')$, with $\dim M' \geq 3$, by fixing a point $x_0 \in M'$, then its action is linearizable near $x_0$, \textit{i.e.} there exists a local diffeomorphism $\psi : \mathcal{U} \rightarrow U$, between a neighbourhood $\mathcal{U}$ of $0$ in $T_{x_0}M'$ and a neighbourhood $U$ of $x_0$ in $M'$, that conjugates near $x_0$ the $H'$-action to its isotropy representation.
\end{prop}

In all this section, we note $x_0 \in M$ a global fixed point of the $H$-action and the isotropy representation of $H$ at $x_0$ is denoted by $\rho_{x_0} : h \in H \mapsto T_{x_0}h \in \CO(T_{x_0}M,g_{x_0})$.

\subsection{Describing the isotropy representation}
\label{ss:isotropy}

Any morphism from a connected simple Lie group into an abelian Lie group being trivial, $\rho_{x_0}$ takes values in $\SO(T_{x_0}M,g_{x_0}) \subset \CO(T_{x_0}M,g_{x_0}) = \R_{>0} \times \SO(T_{x_0}M,g_{x_0})$. Thus, the derivative of $\rho_{x_0}$ is an embedding of Lie algebras $f_{x_0} : \h \rightarrow \so(T_{x_0}M,g_{x_0})$ (if not, the isotropy would be trivial and $H$ would act trivially on a neighbourhood of $x_0$, contradicting the faithfulness of the action). Now, we establish the

\begin{lem}
Let $\rho : \so(1,2) \rightarrow \so(1,n-1)$ be a Lie algebra embedding. Then, there exists $g \in \SO(1,n-1)$ such that $\Ad(g) \circ \rho$ is the natural diagonal map
\begin{equation*}
A \in \so(1,2) \mapsto 
\begin{pmatrix}
A & \\
 & 0
\end{pmatrix}
\in \so(1,n-1).
\end{equation*}
\end{lem}

\begin{proof}
Consider $\rho$ as a real representation of $\so(1,2)$ on $\R^{1,n-1}$. By simplicity, $\rho$ is completely reducible and must admit an irreducible faithful subrepresentation $V \subset \R^{1,n-1}$ (with $\dim V \geq 2$). We claim that the induced sub-Lorentzian quadratic form on $V$ is Lorentzian. Indeed, if not, $V$ would be Riemannian or non-negative and degenerate, with a $1$-dimensional kernel. In the first case, we would get a non-trivial morphism from $\so(1,2)$ into the compact Lie algebra $\so(\dim(V))$. In the second, $\rho$ would preserve the kernel which would be a proper invariant subspace of $V$, contradicting the irreducibility.

Thus, we have an orthogonal decomposition $\R^{1,n-1} = V \oplus V^{\perp}$, with $V^{\perp}$ a Riemannian subrepresentation. By the same argument as above, $\rho|_{V^{\perp}} = 0$. Moreover, $\dim V \geq 3$ since $\so(1,1)$ is abelian and $\rho$ is assumed faithful. Now, we claim that the only faithful irreducible representation of $\so(1,2)$ on a finite dimensional space that respects a Lorentz scalar product is $3$-dimensional. 

This comes from the classification of the finite-dimensional irreducible representations of $\sl(2,\R) \simeq \so(1,2)$. Let $(E,F,H)$ be the standard presentation of $\sl(2,\R)$. If $(V_d,\pi_d)$ is the $(d+1)$-dimensional irreducible representation of $\sl(2,\R)$, $\pi_d(H)$ acts diagonally on $V_d$ with eigenvalues $-d,-d+2,\ldots,d-2,d$. In our situation, $\rho(H)|_V$ moreover generates a one-parameter group of Lorentzian isometries of $V$. The elementary following result then gives $V \simeq V_2$.

\begin{sublemma*}
Let $\phi^t \in \SO(1,m-1)$, $m \geq 3$ be a linear flow acting diagonally on $\R^{1,m-1}$ with eigenvalues $\e^{\lambda_1 t}, \ldots, \e^{\lambda_m t}$. Then, up to permutation, $\lambda_1 = -\lambda_2$ and $\lambda_3 = \cdots = \lambda_m = 0$.
\end{sublemma*}

\begin{proof}
Since the morphism $\Ad_{\so(1,n-1} : \SO(1,n-1) \rightarrow \GL(\so(1,n-1))$ is algebraic, it sends $\phi^t$ on an $\R$-split flow of $\GL(\so(1,n-1))$: this means that $\phi^t$ is in a Cartan subgroup of $\SO(1,n-1)$. The results follows from the form of Cartan subspaces of $\so(1,n-1)$ given in Section \ref{ss:holonomy}.
\end{proof}

Finally, $\rho|_V : \so(1,2) \rightarrow \so(V)$ is an isomorphism for dimensional reasons, proving that we can conjugate in $\so(1,n-1)$ so that all the $\rho(A)$, $A \in \so(1,2)$, have the announced form.
\end{proof}

By connectedness of $H$, we have an isometric identification $(T_{x_0}M,g_{x_0}) \simeq \R^{1,n-1}$ such that
\begin{equation*}
\rho_{x_0}(H) =
\left \{
\begin{pmatrix}
A & \\
 & \id
\end{pmatrix}
, \ A \in \SO_0(1,2)
\right \}
\subset \SO_0(1,n-1).
\end{equation*}

\begin{cor}
When there exists a global fixed point, $H$ is (globally) isomorphic to $\PSL(2,\R)$.
\end{cor}

\begin{proof}
The representation $\rho_{x_0}$ provides a surjective Lie group morphism $H \rightarrow \SO_0(1,2) \simeq \PSL(2,\R)$. This morphism is a local diffeomorphism, and then a covering. Therefore, any element in the center of $H$ is in the kernel of $\rho_{x_0}$. By Proposition \ref{prop:linearizability}, such an element acts trivially near $x_0$. Since any conformal diffeomorphism of $M$ acting trivially on a non-empty open set must be globally trivial, $H$ has a trivial center (we assumed the $H$-action faithful).
\end{proof}

Now we have a good understanding of the isotropy $\rho_{x_0}$, we can describe locally the action of (small elements of) $H$ near $x_0$.

\subsection{``Local orbits'' near a fixed point}

In this section, we fix a basis $(e_1,\ldots,e_n)$ of $T_{x_0}M$ such that $g_{x_0}$ reads $-x_1^2 + x_2^2 + \cdots + x_n^2$ and such that the isotropy representation has the form
\begin{equation*}
A \in \SO_0(1,2) \mapsto	
\begin{pmatrix}
A & \\
 & \id
\end{pmatrix}
\in \SO_0(1,n-1).
\end{equation*}
Let $E$ denote $\Span(e_1,e_2,e_3)$. By Proposition \ref{prop:linearizability} there exists $\mathcal{U} \subset \mathcal{U}' \subset E$ and $\mathcal{V} \subset E^{\perp}$ neighbourhoods of the origin, a neighbourhood $W$ of $x_0$ in $M$, a neighbourhood $V_H \subset H$ of the identity and a diffeomorphism $\psi : \mathcal{U}' \times \mathcal{V} \rightarrow W \subset M$ such that $\psi(0,0) = x_0$, $\forall h \in V_H$, $\rho_{x_0}(h) ( \mathcal{U} \times \mathcal{V}) \subset \mathcal{U}' \times \mathcal{V}$ and $\forall (u,v) \in \mathcal{U} \times \mathcal{V}$, $\psi(\rho_{x_0}(h)(u,v)) = h. \psi(u,v)$. Reducing the open sets if necessary, we assume that $\mathcal{U}$, $\mathcal{U}'$ (resp. $\mathcal{V}$) are open balls in $E$ (resp. $E^{\perp}$) with respect to $x_1^2+x_2^2+x_3^2$ (resp. $x_4^2+\cdots+x_n^2$).

For $v \in \mathcal{V}$, we note $U'_v := \psi(\mathcal{U}' \times \{v\})$ and $U_v := \psi(\mathcal{U} \times \{v\})$. The leaves of the foliation $\{U'_v\}_{v \in \mathcal{V}}$ of $W$ are individually ``locally $H$-invariant'', that is $\forall v \in \mathcal{V}$, $V_H. U_v \subset U'_v$. Moreover, each $U_v$ is partitioned in the following way. Note $q = -x_1^2 + x_2^2 + x_3^2$ the quadratic form induced by $g_{x_0}$ on $E$. In $U_v$, we define for $\lambda \leq 0$
\begin{align*}
\mathcal{O}_+(v,\lambda) & = \psi((\mathcal{U}\cap \{q=\lambda\} \cap \{x_1 > 0\}) \times \{v \}) \text{ and } \\
\mathcal{O}_-(v,\lambda) & = \psi((\mathcal{U}\cap \{q=\lambda\} \cap \{x_1 < 0\}) \times \{v \})
\end{align*}
and for $\lambda > 0$
\begin{align*}
\mathcal{O}(v,\lambda) & = \psi((\mathcal{U}\cap \{q=\lambda\}) \times \{v \}).
\end{align*}
For all $v \in \mathcal{V}$, the leaf $H_v$ splits into the disjoint union
\begin{equation*}
U_v = \{\psi(0,v)\} \cup \bigcup_{\lambda <0} \mathcal{O}_+(v,\lambda) \cup \mathcal{O}_+(v,0) \cup \bigcup_{\lambda > 0}\mathcal{O}(v,\lambda) \cup \mathcal{O}_-(v,0) \cup \bigcup_{\lambda <0} \mathcal{O}_-(v,\lambda).
\end{equation*}
(see Figure \ref{fig:local_orbits}).

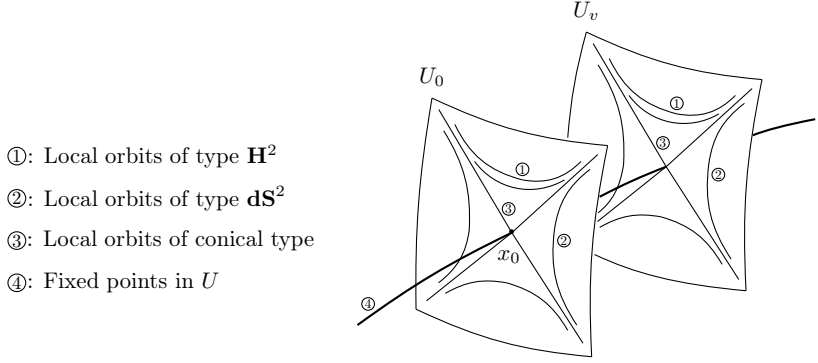
\begin{figure}
\hspace*{-4cm}
\begin{tikzpicture}[scale=.7]
\coordinate (1) at (-2.5,2.1) ;
\coordinate (2) at (-2.8,-1.9);
\coordinate (3) at (0.5,-2.8);
\coordinate (4) at (0.8,1.2) ;

\coordinate (5) at (.45,3.4) ;
\coordinate (6) at (0.3,1.4);
\coordinate (7) at (.7,-.75);
\coordinate (8) at (3.45,-1.5) ;
\coordinate (9) at (3.75,2.5) ;

\coordinate (10) at (-1,-.45) ;

\coordinate (11) at (-3.9,-2.2) ;

\coordinate (12) at (1.9,.8) ;

\coordinate (13) at (.65,.23) ;

\coordinate (14) at (3.2,1.27) ;

\coordinate (15) at (4.7,1.7) ;

\draw (1)to[bend right=5](2) ;
\draw (2)to[bend right=10](3) ;
\draw (1)to[bend right=10](4) ;
\draw (3)to[bend left=7](4) ;

\draw (10) node[scale=.5]{$\bullet$} ;

\draw[thick] (11)to[bend left=5](10) ;


\draw[thick] (13)to[bend left=3](12) ;

\draw[thick, shorten <= 7pt] (14)to[bend left=5](15) ;

\node[scale=.85] at (-1.05,-.9) {$x_0$} ;


\draw[shorten >= 5pt] (10)to(1) ;
\draw[shorten >= 5pt] (10)to(4) ;
\draw[shorten >= 5pt] (10)to(2) ;
\draw[shorten >= 6pt] (10)to(3) ;

\node[scale=.85] at (-2.5,2.5) {$U_0$} ;

\draw (-2.07,1.6)to[bend right=25](-.8,0.55) ;
\draw (-.8,0.55)to[bend right=28](0.3,0.9) ;


\draw (-1.7,.9)to[bend right=42](-0.01,0.55) ;

\coordinate (16) at (-1.1,-1.4) ;
\draw (-2.25,-1.6)to[bend left=20](16) ;
\draw (16)to[bend left=22](0.05,-2.3) ;

\coordinate (17) at (-.2,-.4) ;
\draw (0.22,-2.2)to[bend left=20](17) ;
\draw (17)to[bend left=22](0.47,0.75) ;

\coordinate (18) at (-1.8,-.4) ;
\draw (-2.28,1.48)to[bend left=17](18) ;
\draw (18)to[bend left=22](-2.377,-1.4) ;

\node[scale=.5, draw,circle,inner sep=.5pt] at (-.8,.75) {1};
\node[scale=.5, draw,circle,inner sep=.5pt] at (-.02,-.6) {2};
\node[scale=.5, draw,circle,inner sep=.5pt] at (-1.06,0) {3};

\node[scale=.5, draw,circle,inner sep=.5pt] at (-3.7,-1.8) {4};


\begin{scope} [shift={(2.9,1.25)}]
\coordinate (1) at (-2.5,2.1) ;
\coordinate (2) at (-2.8,-1.9);
\coordinate (3) at (0.5,-2.8);
\coordinate (4) at (0.8,1.2) ;

\draw (1)to[bend right=5](-2.85,0.1) ;
\draw[shorten <= 10pt] (2)to[bend right=10](3) ;
\draw (1)to[bend right=10](4) ;
\draw (3)to[bend left=7](4) ;

\node[scale=.85] at (-2.5,2.5) {$U_v$} ;

\draw[shorten >= 5pt] (-1,-.45)to(-2.5,2.1) ;
\draw[shorten >= 5pt] (-1,-.45)to(0.8,1.2) ;
\draw[shorten >= 13pt] (-1,-.45)to(-2.8,-1.9) ;
\draw[shorten >= 5pt] (-1,-.45)to(0.5,-2.8) ;

\draw (-2.07,1.6)to[bend right=25](-.8,0.55) ;
\draw (-.8,0.55)to[bend right=28](0.3,0.9) ;

\draw (-1.7,.9)to[bend right=42](-0.01,0.55) ;

\coordinate (16) at (-1.1,-1.4) ;
\draw (-2.25,-1.6)to[bend left=20](16) ;
\draw (16)to[bend left=22](0.05,-2.3) ;

\coordinate (17) at (-.2,-.4) ;
\draw (0.22,-2.2)to[bend left=20](17) ;
\draw (17)to[bend left=22](0.47,0.75) ;

\coordinate (18) at (-1.8,-.4) ;
\draw (-2.28,1.48)to[bend left=17](18) ;
\draw[shorten >= 2.5pt] (18)to[bend left=22](-2.377,-1.4) ;

\node[scale=.5, draw,circle,inner sep=.5pt] at (-.8,.75) {1};
\node[scale=.5, draw,circle,inner sep=.5pt] at (-.02,-.6) {2};
\node[scale=.5, draw,circle,inner sep=.5pt] at (-1.06,0) {3};
\end{scope}


\begin{scope} [shift={(1,-1)}]
\node[scale=.7, draw,circle,inner sep=.5pt] at (-11.3,2) {1};
\node[scale=.85] at (-8.75,2) {: Local orbits of type $\H^2$};

\node[scale=.7, draw,circle,inner sep=.5pt] at (-11.3,1.2) {2};
\node[scale=.85] at (-8.67,1.2) {: Local orbits of type $\dS^2$};

\node[scale=.7, draw,circle,inner sep=.5pt] at (-11.3,.4) {3};
\node[scale=.85] at (-8.41,.4) {: Local orbits of conical type};

\node[scale=.7, draw,circle,inner sep=.5pt] at (-11.3,-.4) {4};
\node[scale=.85] at (-9.3,-.4) {: Fixed points in $U$};
\end{scope}
\end{tikzpicture}
\caption{Local orbits near $x_0$}
\label{fig:local_orbits}
\end{figure}

\begin{lem}
\label{lem:local_orbit_in_global_orbit}
For every $v \in \mathcal{V}$, $\lambda \leq 0$ and $\mu > 0$, the point $\psi(0,v)$ is a global fixed point of the $H$-action and each one of the subsets $\mathcal{O}_{+/-}(v,\lambda)$ and $\mathcal{O}(v,\mu)$ is contained in a single $H$-orbit.
\end{lem}

\begin{proof}
The point $\psi(0,v)$ is fixed by $V_H$, and consequently by all of $H$ by connectedness.

We only treat the case of $\mathcal{O}_+(v,\lambda)$, $\lambda < 0$, the other ones being similar. The level set $\{q=\lambda\} \cap \{x_1 > 0\}$ is a $2$-dimensional upper hyperboloid of the Lorentzian space $(E,q)$. Since $\mathcal{U}$ has been chosen to be a Euclidian ball in $E$, if non-empty, the intersection of this hyperboloid with $\mathcal{U}$ is connected. Therefore, for any two points $u,u'$ in this intersection, there exists finitely many $g_1,\ldots,g_s \in V_H$ such that $u' = g_s \ldots g_1 . u$ and $g_i \ldots g_1 . u \in \mathcal{U}$ for all $i \leq s$. This shows that $\psi(u',v) = g_s \ldots g_1 . \psi(u,v)$.
\end{proof}

Each one of these sets being an open subset of a global $H$-orbit, we call them \textit{local $H$-orbits}. We distinguish them by calling
\begin{itemize}
\item ``local $H$-orbit of type $\H^2$'' any $\mathcal{O}_{+/-}(v,\lambda)$ with $\lambda < 0$ ;
\item ``local $H$-orbit of type $\dS^2$'' any $\mathcal{O}(v,\lambda)$ with $\lambda > 0$ ;
\item ``local $H$-orbit of conical type'' any $\mathcal{O}_{+/-}(v,0)$.
\end{itemize}
For every $v \in \mathcal{V}$, if $\epsilon$ denotes the radius of the ball $\mathcal{U}$, we define in $U_v$ the rays $\Delta_T^+(v)= \{\psi(t e_1 , v), \ t \in ]0,\epsilon[\}$, $\Delta_T^-(v)= \{\psi(t e_1 , v), \ t \in ]-\epsilon,0[\}$, $\Delta_L^+(v)= \{\psi(t (e_1+e_3) , v), \ t \in ]0,\epsilon[ \}$, $\Delta_L^-(v)= \{\psi(t (e_1+e_3) , v), \ t \in ]-\epsilon,0[ \}$, and the line $\Delta_S(v) = \{\psi(t e_3 , v), \ t \in ]-\epsilon,\epsilon[ \}$.

In $U_v$, every local $H$-orbit of type $\H^2$ or $\dS^2$ meets a unique $\Delta_T^+(v)$, $\Delta_T^-(v)$ or $\Delta_S(v)$ exactly once and $\Delta_L^+(v) \subset \mathcal{O}^+(v,0)$ and $\Delta_L^-(v) \subset \mathcal{O}^-(v,0)$ (same picture than in Minkowski space).

\vspace*{0.2cm}

Finally, in $\SO_0(1,2)$ acting on $(E,q) = \R^{1,2}$, we note $S_e$ the stabilizer of $e_1$, $S_h$ the stabilizer of $e_3$ and $S_u$ the stabilizer of $e_1 + e_3$. We have $S_e \simeq \SO(2)$, $S_h \simeq \SO_0(1,1)$ and $S_u$ is a unipotent subgroup of $\SO_0(1,2)$. If $x \in U$, let $(H_x)_0$ denote the identity component of the stabilizer of $x$.

\begin{lem}
\label{lem:stabilizer_local_orbits}
Let $v \in \mathcal{V}$. Then,
\begin{itemize}
\item for any $x \in \Delta_T^{+/-}(v)$, we have $(H_x)_0 = S_e$ ;
\item for any $x \in \Delta_S(v)$, we have $(H_x)_0 = S_h$ ;
\item for any $x \in \Delta_L^{+/-}(v)$, we have $(H_x)_0 = S_u$.
\end{itemize}
\end{lem}

\begin{proof}
In the three cases, $(H_x)_0$ has dimension $1$ and by the linearization property, it contains a neighbourhood of the identity of $S_e$ (resp. $S_h$, $S_u$) and it is sufficient since $S_e$, $S_h$ and $S_u$ are connected.
\end{proof}

\begin{lem}
\label{lem:equality_local_orbit}
Let $v \in \mathcal{V}$, $\lambda < 0$ and $x \in \mathcal{O}^+(v,\lambda)$. Then we have $(H.x) \cap U = \mathcal{O}^+(v,\lambda)$. The same is true for $\mathcal{O}^-(v,\lambda).$
\end{lem}

\begin{proof}
The inclusion $\supset$ has been proved in Lemma \ref{lem:local_orbit_in_global_orbit}. In particular, we can assume that $x$ belongs to $\Delta^+(v)$. Now observe that the normalizer of $S_e$ in $\SO_0(1,2)$ is $S_e$ itself. Since $S_e$  is the identity component of the stabilizer $H_x$ of $x$ in $H$, it is normalized by $H_x$. Thus, $S_e$ is exactly the stabilizer of $x$.

Let $y \in (H.x) \cap U$. Its stabilizer $H_y$ is conjugated to $S_e$ in $H$ and cannot admit a conjugate of $S_h$ nor $S_u$ as subgroup. By Lemma \ref{lem:stabilizer_local_orbits}, this implies that $y$ belongs to a local $H$-orbit of type $\H^2$: we have $\lambda' < 0$ and $v' \in \mathcal{V}$ such that $y \in \mathcal{O}^+(v',\lambda')$ or $\mathcal{O}^-(v',\lambda')$. In both cases, by Lemma \ref{lem:local_orbit_in_global_orbit} there is $h \in H$ such that $z := h.y \in \Delta_T^+(v')$ or $\Delta_T^-(v')$ and $z$ belongs to the same local $H$-orbit than $y$. In particular its stabilizer is $H_{z} = S_e$. Thus, if $h_0 \in H$ is such that $z = h_0.x$, it normalizes $S_e$. So, we have $h_0 \in S_e$, and then $z = x$, proving that $x$ and $y$ belong to the same local $H$-orbit.
\end{proof}

\begin{cor}
There exists an $H$-orbit $\mathcal{O} \subset F_{\leq 2}$ such that $\bar{\mathcal{O}}$ does not contain any fixed point.
\end{cor}

\begin{proof}
What we have seen before shows that there exists $x \in M$ with stabilizer $S_e$ (we assumed that there exists a fixed point). Let $\mathcal{O}$ be its $H$-orbit. Assume that $\mathcal{O}$ meets the neighbourhood $U$ of $x_0$. By the same argument than in the previous proof, any point of $\mathcal{O} \cap U$ must be on a local $H$-orbit of type $\H^2$. Then, Lemma \ref{lem:equality_local_orbit} says that $\mathcal{O} \cap U$ is a submanifold of $U$ that does not contain $x_0$. This proves that $x_0$ (any fixed point) cannot be in the closure of $\mathcal{O}$. 
\end{proof}

Thus, Proposition \ref{prop:no_fixed_point} is proved if we take $F := \bar{O}$ as in the previous corollary.

\section{Proof of the main theorem}
\label{s:conformal_flatness}

What we have done so far shows that if $H \simeq_{\text{loc}} \PSL(2,\R)$ acts conformally and essentially on $(M,g)$, then there exists an $H$-invariant closed subset $F \subset M$ in which all $H$-orbits have dimension $1$ or $2$. From now on, the Lorentz manifold $(M,g)$ is assumed real-analytic.

\vspace*{0.2cm}

We will easily see that if there exists a $1$-dimensional orbit in $F$, then $H$ contains a flow admitting a singularity on this orbit and which is non-linearizable near its singularity. A result of Frances and Melnick will directly ensure the conformal flatness of $(M,g)$. We postpone this question to Section \ref{ss:orbit_dim1}. So, we can assume that every $H$-orbit in $F$ is $2$-dimensional (it happens for instance when $\SO_0(1,2)$ acts on a Hopf manifold). This situation is treated in the next section and is the core of the proof.

\subsection{Conformal flatness near degenerate $2$-dimensional orbits}
\label{ss:orbit_dim2}

Let $(X,Y,Z)$ be an $\sl(2)$-triple in $\h$ and let $S < H$ be the connected subgroup with Lie algebra $\Span(X,Y) \simeq \aff(\R)$. 

\vspace*{0.2cm}

Since $F$ is $S$-invariant, Proposition \ref{prop:zimmer_conformal} applies and gives a point $x_0 \in F$ such that $\Ad(S)\h_{x_0} \subset \h_{x_0}$ and the induced action $\bar{\Ad}(S) \curvearrowright \h / \h_{x_0}$ is conformal with respect to the quadratic form $q_{x_0}$ obtained by restricting the metric $g_{x_0}$ to $T_{x_0}(H.x_0) \simeq \h / \h_{x_0}$. Because $x_0 \in F$, $\h_{x_0}$ is a line in $\h$ and we must have $\h_{x_0} = \R.Y$. Since $\bar{\Ad}(\e^{tY})\bar{X} = \bar{X}$ and $\bar{\Ad}(\e^{tY})\bar{Z} = \bar{Z} + t \bar{X}$, we see that $q_{x_0}$ is necessarily degenerate, with $\ker q_{x_0} = \R.X (\text{ mod. } \h_{x_0})$. Thus, the orbit $H.x_0$ is degenerate and at $x_0$, the vector field $X$ gives the direction of the kernel. In particular, the integral curve $\{\phi_X^t(x_0)\}$ is an immersed light-like curve.	

\vspace*{0.2cm}

The following proposition says that moreover, it is possible to choose the point $x_0$ so that there exists a local conformal vector field defined near $x_0$, with a singularity at $x_0$ and whose local flow \textit{does not come from the action of $H$}. This flow will play a central role in proving the conformal flatness. Its existence is derived from a result of Gromov's theory of rigid geometric structures (\cite{gromov}), and relies on the analyticity of $(M,g)$.

\begin{prop}
\label{prop:local_vector_field}
There exists a point $x_0 \in F$ such that the conclusions of Proposition \ref{prop:zimmer_conformal} remain valid and such that there exists a local conformal vector field $A^*$ defined on a neighbourhood $U$ of $x_0$, such that $A^*(x_0) = 0$, $[A^*,X] = 0$ on $U$ and whose local flow $\phi_{A^*}^t$ is defined for $t \geq 0$ on all of $U$ and is linearizable on $U$. Precisely, there exists a chart $\psi : \mathcal{U} \rightarrow U$, where $\mathcal{U} \subset T_{x_0}M$ is a neighbourhood of the origin, such that $T_0 \psi = \id_{T_{x_0}M}$ and there are coordinates $(u_1,\ldots,u_n)$ on $T_{x_0}M$ in which the metric $g_{x_0}$ reads $2u_1u_n + u_2^2 + \cdots + u_{n-1}^2$, with $X_{x_0} = \frac{\partial}{\partial u_1}(x_0)$ and such that
\begin{equation*}
\psi^{-1} \phi_{A^*}^t \psi =
\begin{pmatrix}
1 & & & & \\
 & \e^{-t} & & & \\
 & & \ddots & & \\
 & & & \e^{-t} & \\
 & & & & \e^{-2t}
\end{pmatrix}
\text{ for all } t \geq 0.
\end{equation*}
\end{prop}

\begin{figure}
\begin{tikzpicture}[scale=.6]
\coordinate (1) at (-2.2,2) ;
\coordinate (2) at (-2.5,-2);
\coordinate (3) at (0.5,-2.8);
\coordinate (4) at (0.8,1.2) ;

\coordinate (5) at (0.8,3.2) ;
\coordinate (6) at (0.3,1.25);
\coordinate (7) at (0.45,-.7);
\coordinate (8) at (3,-1.2) ;
\coordinate (9) at (3.5,2.6) ;

\coordinate (10) at (-1.1,-.45) ;

\coordinate (11) at (-3.9,-2.4) ;

\coordinate (12) at (1.67,.75) ;

\coordinate (13) at (.62,.32) ;

\coordinate (14) at (3.2,1.27) ;

\coordinate (15) at (4.7,1.7) ;

\coordinate (16) at (2.85,1.9) ;

\draw (1)to[bend right=10](2) ;
\draw (2)to[bend right=15](3) ;
\draw (1)to[bend right=15](4) ;
\draw (3)to[bend left=12](4) ;

\draw (5)to[bend right=8](6) ;
\draw (7)to[bend right=15](8) ;
\draw (8)to[bend left=10](9) ;
\draw (5)to[bend right=10](9) ;

\draw (10) node[scale=.5]{$\bullet$} ;

\draw (16) node[scale=.5]{$\bullet$} ;

\draw[thick] (11)to[bend left=10](10) ;

\draw (12) node[scale=.5]{$\bullet$} ;

\draw[thick] (13)to[bend left=3](12) ;

\draw[thick] (14)to[bend left=5](15) ;

\node[scale=.8] at (-4.3,-2.45) {$\Delta$} ;

\node[scale=.8] at (-1,-.8) {$x_0$} ;

\node[scale=.8] at (2.1,.3) {$\Delta(x)$} ;

\node[scale=.8] at (-2.2,2.27) {$\mathcal{F}_0$};

\node[scale=.8] at (1,3.45) {$\mathcal{F}_{u_1}$};

\node[scale=.8] at (2.88,2.2) {$x$};

\coordinate (17) at (-1,1.2) ;
\coordinate (18) at (-1.3,0.2) ;

\coordinate (19) at (0.3,.4) ;
\coordinate (20) at (-.6,0.2) ;

\coordinate (21) at (.2,-1.5) ;
\coordinate (22) at (-.5,-.5) ;

\coordinate (23) at (-2.3,-1.6) ;
\coordinate (24) at (-1.28,-1) ;
\coordinate (25) at (-1.4,-1.2) ;

\coordinate (26) at (-2.3,0.6) ;
\coordinate (27) at (-1.8,-.3) ;

\coordinate (28) at (-.5,-2.3) ;
\coordinate (29) at (-.6,-1.2) ;
\coordinate (30) at (-.72,-1) ;

\coordinate (31) at (1.9,1.2) ;

\draw[->] (17)to[bend right=18](18) ;
\draw (18)to[bend right=5](10) ;
\draw[->] (19)to[bend right=18](20) ;
\draw (20)to[bend right=10](10) ;
\draw[->] (21)to[bend right=40](22) ;
\draw (22)to[bend right=10](10) ;
\draw[->] (23)to[bend right=15](25) ;
\draw (25)to[bend right=10](24) ;
\draw[->] (26)to[bend right=25](27) ;
\draw (27)to[bend right=15](10) ;
\draw[->] (28)to[bend right=15](29) ;
\draw (29)to[bend right=7](30) ;

\draw[->] (16)to[bend right=15](31) ;
\draw (31)to[bend right=7](12) ;

\draw[dashed, very thin] (-1.1,-.45)to[bend left=2](.56,.3) ;
\draw[dashed, very thin] (1.67,.75)to[bend left=2](3,1.2) ;
\end{tikzpicture}
\caption{Dynamic of the flow $\phi_{A^*}^t$}
\label{fig:dynamic_flow}
\end{figure}
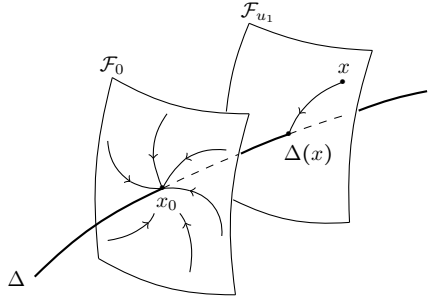

The proof of this proposition is based on the same kind of arguments than the one of Proposition \ref{prop:zimmer_conformal}. It is however longer and we will need an additional algebraic work in the Lie algebra $\so(2,n)$. We postpone it to Section \ref{s:proof_local_flow} and enter directly in the core of the proof of Theorem \ref{thm:main}.

\subsubsection{Vanishing of the Weyl tensor near $x_0$}

In fact, the open set $\mathcal{U} \subset T_{x_0}M$ in Proposition \ref{prop:local_vector_field} is a cylinder $]-\epsilon , \epsilon [ \times B$, where $B$ is some Euclidian ball in the hyperplan $\{u_1 = 0\}$. The curve $\Delta = \{\psi(u_1,0,\ldots,0), \ |u_1|<\epsilon\}$ is exactly the set of fixed points of $\{\phi_{A^*}^t\}_{t \geq 0}$ on $U$. If we note $\mathcal{F}_{u_1} = \psi(\{u_1\} \times B)$, so that $U$ is foliated by these hypersurfaces, $\phi_{A^*}^t$ preserves every $\mathcal{F}_{u_1}$ and for all $x \in U$, we have $\phi_{A^*}^t(x) \rightarrow \Delta(x)$, where $\Delta(x)$ is the point at the intersection of $\Delta$ and the leaf containing $x$. Moreover, since $A^*$ commutes with $X$ and vanishes at $x_0$, its flow must fix each point of the curve $\{\phi_X^s(x_0)\}$, showing that $\Delta$ is a piece of the orbit $\{\phi_X^s(x_0)\}$ (see Figure \ref{fig:dynamic_flow}).

\vspace*{0.2cm}

In what follows, we note $\partial_i$ the vector field $\frac{\partial}{\partial u_i}$ on $U$ given by the coordinates $(u_1,\ldots,u_n)$. Let $\mathcal{H}$ be the field of hyperplanes on $U$ given by $\mathcal{H}_x = \Span (\partial_2(x),\ldots,\partial_n(x))$. We also note $\phi^t = \phi_{A^*}^t$.

\begin{lem}
\label{lem:signature_hyperplane}
Reducing $\mathcal{U}$ and $U$ if necessary, the field $\mathcal{H}$ is degenerate, with kernel $\R.\partial_n$.
\end{lem}

\begin{proof}
Since the field of $(n-2)$-planes $\Span(\partial_2,\ldots,\partial_{n-1})$ is Euclidian at $x_0$, shrinking $\mathcal{U}$ if necessary, it is also Euclidian on all of $U$. Let $\lambda(x,t) > 0$ be such that $[(\phi^t)^* g]_x = \lambda(x,t) g_x$ for all $x \in U$ and $t \geq 0$. Let $x^t$ denote $\phi^t(x)$ for $t \geq 0$ and $x \in U$. Since $g(\partial_2,\partial_2) > 0$ on $U$, from $\lambda(x,t) g_x(\partial_2,\partial_2) = \e^{-2t} g_{x^t}(\partial_2,\partial_2)$ we deduce that for all $x$, $\e^{2t} \lambda(x,t)$ has a positive limit when $t \rightarrow \infty$. For $2 \leq i \leq n-1$, we also have $\lambda(x,t) g_x(\partial_n,\partial_i) = \e^{-3t} g_{x^t}(\partial_n,\partial_i)$ and $\lambda(x,t) g_x(\partial_n,\partial_n) = \e^{-4t} g_{x^t}(\partial_n,\partial_n)$. Therefore, $\partial_n(x)$ is isotropic and orthogonal to $\mathcal{H}_x$.
\end{proof}

We treat the case $\dim M \geq 4$, and postpone the $3$-dimensional situation at the end of this section. Let $W$ denote the $(3,1)$-Weyl tensor of $(M,g)$ (\cite{besse}, 1.117). It is conformally invariant and it detects the conformal flatness: for all open subset $U \subset M$, $W|_U \equiv 0$ if and only if $U$ is conformally flat (\cite{besse}, 1.159, 1.165). For all $x \in U$ and subspace $E \subset T_xM$, we adopt the notations $W_x(E,E,E) = \{W_x(u,v,w), \ u,v,w \in E\}$ and $\Ima W_x = W_x(T_xM,T_xM,T_xM)$. We fix $\|.\|_x$ an auxiliary Riemannian metric on $U$ and note $x^t := \phi^t(x)$ for all $x \in U$ and $t \geq 0$.

\begin{lem}
\label{lem:contraction_rates}
Let $x \in U$, $v \in T_xM$ and $f(t)$ denote $\| (\phi^t)_* v \|_{x^t}$. If $f(t) \rightarrow 0$, then $v \in \mathcal{H}_x$, if $f(t) = o(\e^{-t})$, then $v$ is a multiple of $\partial_n(x)$ and if $f(t) = o(\e^{-2t})$, then $v=0$.
\end{lem}

\begin{proof}
We have a framing $(\partial_1,\ldots,\partial_n)$ of $TU$. If we decompose $v = \lambda_1 \partial_1(x) + \cdots + \lambda_n \partial_n(x)$, then we have $(\phi^t)_* v = \lambda_1 \partial_1(x^t) + \e^{-t} (\lambda_2\partial_2(x^t) + \cdots + \lambda_{n-1}\partial_{n-1}(x^t)) + \e^{-2t} \lambda_n \partial_n(x^t)$. Thus, $|\lambda_1| \|\partial_1(x^t)\|_{x^t} \leq f(t) + C \e^{-t}$, $C \geq 0$, and $f(t) \rightarrow 0 \Rightarrow \lambda_1 = 0$. The other cases follow from the same reasoning.
\end{proof}

\begin{lem}
For all $x \in U$, $W_x(\mathcal{H}_x,\mathcal{H}_x,\mathcal{H}_x) = 0$ and $\Ima W_x \subset \mathcal{H}_x$.
\end{lem}

\begin{proof}
If $i,j,k \geq 2$, then $(\phi^t)_* W_x(\partial_i,\partial_j,\partial_k) = W_{x^t}((\phi^t)_* \partial_i,(\phi^t)_* \partial_j,(\phi^t)_* \partial_k) = \e^{-kt} W_{x^t}(\partial_i,\partial_j,\partial_k)$ for some integer $k \geq 3$, by conformal invariance of $W$. Thus, $\| (\phi^t)_* W_x(\partial_i,\partial_j,\partial_k) \|_{x^t} = O(\e^{-3t})$, and by Lemma \ref{lem:contraction_rates}, we obtain $W_x(\partial_i,\partial_j,\partial_k) = 0$, proving the first point.

If $u,v,w \in T_xM$, then $W_x(u,v,w)$ is a linear combination of the $W_x(\partial_i,\partial_j,\partial_k)$'s, with $1 \leq i,j,k \leq n$ and at least one index greater than $1$ since $W_x(\partial_1,\partial_1,\partial_1) = 0$. This proves that $\|(\phi^t)_* W_x(u,v,w) \|_{x^t} = O(\e^{-t})$, and applying Lemma \ref{lem:contraction_rates}, we obtain $W_x(u,v,w) \in \mathcal{H}_x$.
\end{proof}

\begin{lem}
\label{lem:two_hyperplanes}
Let $x \in U$ and $\mathcal{H}_x'$ be a degenerate hyperplane of $T_xM$. Assume $W_x(\mathcal{H}_x',\mathcal{H}_x',\mathcal{H}_x') = 0$ and $\Ima W_x \subset \mathcal{H}_x'$. Then, $\mathcal{H}_x \neq \mathcal{H}_x' \Rightarrow W_x = 0$.
\end{lem}

\begin{proof}
Let $e_1,e_n \in T_xM$ be isotropic and such that $\mathcal{H}_x = e_1^{\perp}$ and $\mathcal{H}_x' = e_n^{\perp}$. If $e_1$ and $e_n$ are non-proportional, then the plane $\Span(e_1,e_n)$ is Lorentzian and $F:= \Span(e_1,e_n)^{\perp}$ is Riemannian. Thus, $\mathcal{H}_x = \R.e_1 \oplus F$ and $\mathcal{H}_x' = F \oplus \R.e_n$. Since $W_x$ ``vanishes'' in restriction to these hyperplanes, by multi-linearity and since $W_x$ is skew-symmetric with respect to its first two entries, we are left to prove that for all $u \in F$ 
\begin{equation*}
W_x(e_1,e_n,e_1) = W_x(e_n,e_1,e_n) = W_x(e_1,u,e_n) = W_x(e_n,u,e_1) = W_x(e_1,e_n,u) = 0.
\end{equation*}
By hypothesis, $\Ima W_x \subset \mathcal{H}_x \cap \mathcal{H}_x' = F$. By construction of the Weyl tensor, the $(4,0)$-tensor $g_x(W(.,.,.),.)$ is a component of the $(4,0)$-Riemann curvature tensor of $g$ with respect to a standard decomposition of the subspace of $S^2(\Lambda^2 T_x^*M)$ consisting of tensors satisfying the (first) Bianchi identity (see \cite{besse}, 1.117). In particular, it satisfies the same symmetries than a $(4,0)$-Riemann curvature tensor. This is enough to conclude: for instance, for all $v \in F$ we have $g_x(W_x(e_1,e_n,e_1),v) = g_x(W_x(e_1,v,e_1),e_n) = 0$ since $e_1, v \in \mathcal{H}_x$. Therefore, $W_x(e_1,e_n,e_1) \in F \cap F^{\perp} = 0$. Similarly, the four others are also orthogonal to $F$.
\end{proof}

\begin{cor}
The Weyl tensor vanishes in restriction to $\Delta$.
\end{cor}

\begin{proof}
It is enough to prove that $W_{x_0} = 0$ since $\Delta$ is a piece of the orbit $\{\phi_X^t(x_0)\}$. Remark that the properties of $\mathcal{H}_x$ we have exhibited are conformal: if $f \in \Conf(T_xM,g_x)$, then $\mathcal{H}_x':=f(\mathcal{H}_x)$ is degenerate, and satisfies the hypothesis of the previous lemma by conformal invariance of $W$.

Assume that $W_{x_0} \neq 0$. Then, $\mathcal{H}_{x_0}$ is preserved by $\Conf(T_{x_0}M,g_{x_0})$ by Lemma \ref{lem:two_hyperplanes}. Now, consider the flow of the conformal vector field $Y$. Since $\h_{x_0} = \R.Y$, we have $T_{x_0} \phi_Y^t \in \Conf(T_{x_0}M,g_{x_0})$, and $T_{x_0} \phi_Y^t$ must preserve $\mathcal{H}_{x_0}$. But it also preserves the tangent space of the orbit $T_{x_0}(H.x_0)$. Since $\partial_1(x_0) = X_{x_0}$, we have $T_{x_0}(H.x_0) \nsubset \mathcal{H}_{x_0}$. Consequently, $T_{x_0}(H.x_0) \cap \mathcal{H}_{x_0}$ is a Riemannian line in $T_{x_0}(H.x_0)$ which is preserved by $T_{x_0} \phi_Y^t$. Here is the contradiction: the action of $T_{x_0} \phi_Y^t$ on $T_{x_0}(H.x_0)$ is conjugated to the action $\bar{\Ad}(\e^{tY})$ on $\h / \h_{x_0}$ and the only line in $\h / \h_{x_0}$ which is preserved by $\bar{\Ad}(\e^{tY})$ is $\R.X$ modulo $\h_{x_0}$, and it is isotropic (cf. beginning of Section \ref{s:conformal_flatness}).
\end{proof}

\begin{cor}
The Weyl tensor vanishes in restriction to $U$.
\end{cor}

By analyticity, $W$ must then vanish on all of $M$ by connectedness and this corollary finishes the proof of Theorem \ref{thm:main} in the case $\dim M \geq 4$.

\begin{proof}
This lemma is also based on the dynamical properties of the flow $\phi^t$. For all $1 \leq i,j,k \leq n$, we now have $\|W_{x^t}(\partial_i,\partial_j,\partial_k)\|_{x^t} \rightarrow 0$, since $W|_{\Delta} \equiv 0$. Moreover, $\|(\phi^t)_*W_x(\partial_i,\partial_j,\partial_k)\|_{x^t} = \e^{-kt} \|W_{x^t}(\partial_i,\partial_j,\partial_k)\|_{x^t}$ with $k \geq 2$ as soon as two indexes among $i,j,k$ are greater than $1$. For such $i,j,k$'s, we have $\|(\phi^t)_*W_x(\partial_i,\partial_j,\partial_k)\|_{x^t} = o(\e^{-2t})$ and Lemma \ref{lem:contraction_rates} gives $W_x(\partial_i,\partial_j,\partial_k) = 0$.

Since $W_x$ is skew-symmetric with respect to its first two entries, we are left to consider the $W_x(\partial_1,\partial_i,\partial_1)$'s, for $2 \leq i \leq n$. We also have $\|(\phi^t)_*W_x(\partial_1,\partial_n,\partial_1)\|_{x^t} = o(\e^{-2t})$, since $(\phi^t)_* \partial_n(x) = \e^{-2t}\partial_n(x^t)$. Thus, $W_x(\partial_1,\partial_n,\partial_1) = 0$. 

For $2 \leq i \leq n-1$, we obtain $\|(\phi^t)_*W_x(\partial_1,\partial_i,\partial_1)\|_{x^t} = o(\e^{-t})$ and Lemma \ref{lem:contraction_rates} gives that $W_x(\partial_1,\partial_i,\partial_1)$ is proportional to $\partial_n(x)$. We conclude by using the symmetries of $W$: if $g$ is a metric in the conformal class, we have
\begin{equation*}
g_x(W_x(\partial_1,\partial_i,\partial_1),\partial_1) = g_x(W_x(\partial_1,\partial_1,\partial_1),\partial_i) = 0.
\end{equation*}
Thus, $W_x(\partial_1,\partial_i,\partial_1) \perp \partial_1(x)$. But $W_x(\partial_1,\partial_i,\partial_1)$ is a multiple of $\partial_n(x)$, so it is also orthogonal to $\mathcal{H}_x$. Finally, it is orthogonal to $T_xM = \R.\partial_1(x) \oplus \mathcal{H}_x$ and must be zero.
\end{proof}

\subsubsection{Vanishing of the Cotton tensor in dimension $3$}

When $M$ is $3$-dimensional, the Weyl tensor always vanishes and one has to consider the $(3,0)$-Cotton tensor to detect the conformal flatness of $(M,g)$. It is defined as follows: if $P =\text{Ric} - \frac{1}{4}Sg$ denotes the Schouten tensor of $(M,g)$, then define $C_x(u,v,w) = (\nabla_u P)_x(v,w) - (\nabla_v P)_x(u,w)$ for all $x \in M$ and $u,v,w \in T_xM$. Then, $C$ is conformally invariant, \textit{i.e.} two conformal metrics $g$ and $g'$ give rise to the same Cotton tensor (\cite{juhl}, p.393), and for all $U \subset M$, $C|_U \equiv 0$ if and only if $U$ is conformally flat (\cite{lafontaine}, C.6).

We take back the notations of the previous section. For $t \geq 0$, by conformal invariance of $C$,
\begin{equation*}
C_{\phi^t(x)}((\phi^t)_*\partial_i,(\phi^t)_* \partial_j,(\phi^t)_* \partial_k) = C_x(\partial_i,\partial_j,\partial_k),
\end{equation*}
for all $i,j,k \in \{1,2,3\}$ and $x \in U$. Since $C$ is skew-symmetric in its two first entries, we only consider $i \neq j$. Therefore, we have $k \geq 1$ such that $C_x(\partial_i,\partial_j,\partial_k) = \e^{-kt} C_{\phi^t(x)}(\partial_i,\partial_j,\partial_k)$. And because $C_{\phi^t(x)}(\partial_i,\partial_j,\partial_k)$ is bounded near $\Delta(x) = \lim \phi^t(x)$, we have $C_x(\partial_i,\partial_j,\partial_k) = 0$, proving that $C$ vanishes on a neighbourhood of $x_0$, and by analyticity of $(M,g)$ on all of $M$.

\subsection{Orbits of dimension $1$}
\label{ss:orbit_dim1}

We prove here that if there exists a $1$-dimensional $H$-orbit, then $(M,g)$ is conformally flat. We reuse the analyticity assumption, but let us mention that with a more elaborated work, it is possible to prove that in smooth regularity, an open set containing the orbit in its closure is conformally flat. Recall the

\begin{thm}[\cite{frances_melnick13}, Th. 1.2]
\label{thm:nonlinear_conformal_flow}
Let $(M,g)$ be an analytic Lorentz manifold of dimension greater than or equal to $3$. Let $X$ be an analytic conformal vector field on $M$, admitting a singularity at $x_0 \in M$. If $X$ is not analytically linearizable near $x_0$, then $(M,g)$ is conformally flat.
\end{thm}

Recall that we note $\h_x$ the Lie algebra of the stabilizer of a point $x$. Let $x_0 \in M$ be such that $\h_{x_0}$ has codimension $1$ in $\h$. Let $(X,Y,Z)$ be an $\sl(2)$-triple of $\h$. Since any $2$-dimensional Lie subalgebra of $\sl(2,\R)$ is conjugated to the affine Lie algebra, replacing $x_0$ by some $h.x_0$ if necessary, we must have $\h_{x_0} = \Span(X,Y)$. We claim that $Y$ is not linearizable near $x_0$, what will be sufficient with Theorem \ref{thm:nonlinear_conformal_flow}.

Since $Z$ is non-singular at $x_0$, we can choose coordinates $(u_1,\ldots,u_n)$ near $x_0$ so that $Z = \frac{\partial}{\partial u_1}$ on this neighbourhood. Thus, for small $t$, we have $\phi_Z^t(x_0) = (t,0,\ldots,0) =: \alpha(t)$. Since the orbit has dimension $1$, there is a neighbourhood of the identity $V \subset H$ such that $V.x_0 \subset \{\alpha(t), \ t \text{ small} \}$, and $X$ and $Y$ must be tangent to the curve $\alpha$. In restriction to this curve, write $X = f(t) \frac{\partial}{\partial u_1}$ and $Y = g(t) \frac{\partial}{\partial u_1}$. From the relation $[Z,X] = -X$, we get $f(t) = -t$ and from the relation $[Z,Y] = -X$ we get $g(t) = -t^2 / 2$. It is now elementary to compute the action of the flow of $Y$ on the curve $\alpha$:
\begin{equation*}
\phi_Y^t(\alpha(u_1)) = \alpha \left ( \frac{u_1}{1-\frac{t}{2}u_1} \right ),
\end{equation*}
whenever $\alpha	(u_1)$ is in the domain of the chart and the flow is defined at time $t$. Thus, $\phi_Y^t$ has a parabolic action on the curve $\alpha$. In \cite{frances_localdynamics}, Lemma 4.6, Frances had already proved that such a flow cannot be linearizable near its singularity.

\section{Existence of a local conformal vector field: Proof of Proposition \ref{prop:local_vector_field}}
\label{s:proof_local_flow}

Let $\pi : \hat{M} \rightarrow M$ and $\omega$ be the Cartan bundle and the Cartan connection defined by $(M,[g])$. We still note $G = \PO(2,n)$, and $P < G$ the stabilizer of a point in $\Ein^{1,n-1}$, \textit{i.e.} the stabilizer of an isotropic line in $\R^{2,n}$.

\subsection{Integrability of Killing generators}

The key ingredient in the proof of Proposition \ref{prop:local_vector_field} is a formulation for Cartan geometries of a ``Frobenius theorem'', as Gromov named it in \cite{gromov}. More recently, Melnick proposed a formulation and a proof of it for real-analytic Cartan geometries in \cite{melnick}, and we gave another approach in \cite{article1} for smooth Cartan geometries. Roughly speaking, the idea is to give necessary and sufficient conditions on a tangent vector so that it can be locally extended into a conformal vector field.

Generalizing the tensor curvature and the covariant derivative of affine geometry, there is a notion of curvature and of covariant derivative for Cartan geometries. In order to avoid non-necessary technicality, we do not give precise definitions and refer the reader to \cite{melnick}, §3 and \cite{article1}, §2.1, §2.3.1. Let $r$ be a positive integer. Using the curvature and its $r$ first covariant derivatives, we build a $P$-equivariant map
\begin{equation*}
\mathcal{D}^r K : \hat{M} \rightarrow \mathcal{W}^r,
\end{equation*}
where $\mathcal{W}^r$ is a finite-dimensional vector space, endowed with a linear right-action of $\Ad_{\g}(P)$. This map is such that if $f : U \rightarrow V$ is a local conformal diffeomorphism, then $\mathcal{D}^r K \circ \hat{f} = \mathcal{D}^r K$ on $\pi^{-1}(U)$. Thus, if $\hat{X}$ is the lift of a local conformal Killing vector field defined near $x$, we have $(\hat{X}. \mathcal{D}^r K)(\hx) = 0$ for all $\hx$ lying over $x$. For compact real-analytic Cartan geometries, if $r$ is large enough, the ``Frobenius theorem'' says that the converse property holds everywhere. Let us state it precisely. If $A \in \g$, let $\tilde{A}$ denote the $\omega$-constant vector field $\omega^{-1}(A) \in \mathfrak{X}(\hat{M})$.

\begin{dfn}
Let $\hx \in \hat{M}$ and $r > 0$. A Killing generator of order $r$ at $\hx$ is an element $A \in \g$ such that $(\tilde{A}.\mathcal{D}^{r-1}K)(\hx)=0$. We note $\Kill^{r}(\hx)$ the space of Killing generators of order $r$ at $\hx$.
\end{dfn}

\begin{thm}[\cite{melnick}, Theorem 3.11]
\label{thm:frobenius}
Let $(M,\hat{M},\omega)$ be a real-analytic, compact Cartan geometry. Then, there exists $r_0 \geq 0$ such that for all $\hx \in \hat{M}$ and $A \in \Kill^{r_0+1}(\hx)$, there exists a local Killing vector field $A^*$ defined near $x \in M$ and such that $\omega_{\hx}(\hat{A}^*) = A$.
\end{thm}

\begin{rem}
In smooth regularity, this result remains valid but only over an open and dense subset of $M$, called the integrability domain of the Cartan geometry $(M,\hat{M},\omega)$ (cf \cite{article1}, Theorem 2).
\end{rem}

We can start the proof of the existence of the flow $\phi_{A^*}^t$. Recall that we have a conformal action of $H \simeq_{\text{loc}} \PSL(2,\R)$ on $(M,g)$ such that there exists a closed $H$-invariant subset $F \subset M$ in which every orbit is $2$-dimensional. Let $r_0$ be the integer given by Theorem \ref{thm:frobenius}. 

As in the case of Proposition \ref{prop:zimmer_conformal}, the beginning of this proof is directly inspired by more general results that we have adapted to our context. The idea is to apply Borel's density in a larger space, in order to get informations on the curvature. This is done in \cite{bader_frances_melnick}, and also \cite{melnick} where Melnick uses this method to give a proof of Gromov's centralizer theorem for Cartan geometries.

\subsection{Applying Borel's density theorem} 
Recall that the action of $H$ defines an $H \times P$-equivariant map $\iota : \hat{M} \rightarrow \Mon(\h,\g)$, where $\Mon(\h,\g)$ denotes the variety of linear injective maps $\h \rightarrow \g$ (proof of Proposition \ref{prop:zimmer_conformal}). Let $V = \Mon(\h,\g) \times \mathcal{W}^{r_0}$, let $\Ad_{\h}(H) \times \Ad_{\g}(P)$ act on it via $(\Ad_{\h}(h),\Ad_{\g}(p)).(\alpha,w) = (\Ad_{\g}(p) \circ \alpha \circ \Ad_{\h}(h^{-1}),w.\Ad_{\g}(p^{-1}))$ and let $\phi = (\iota,\mathcal{D}^{r_0}K) : \hat{M} \rightarrow V$. Immediately, we get that $\phi$ is $H \times P$-equivariant when $H \times P$ act on $\hat{M}$ via $h.\hx.p^{-1}$, since for all $h \in H$, we have $\mathcal{D}^{r_0}K(h.\hx) = \mathcal{D}^{r_0}K(\hx)$.

\vspace*{0.2cm}

Let $(X,Y,Z)$ be an $\sl(2)$-triple of $\h$. Let $S < H$ be the connected Lie subgroup with Lie algebra $\Span(X,Y) \simeq \aff(\R)$. Since $S$ is amenable, there exists a finite $S$-invariant measure $\mu$ supported on $F$. Using exactly the same methods than in the proof of Proposition \ref{prop:zimmer_conformal}, by Borel's density theorem we obtain that for $\mu$-almost every $x \in M$, for every $\hx$ over $x$, we have $\bar{S}.\phi(\hx) \subset \bar{P}.\phi(\hx)$ where $\bar{S}$ denotes the Zariski closure of $\Ad_{\h}(S)$ in $\GL(\h)$ (it is isomorphic to the affine group $\R^* \ltimes \R$ of the real line) and $\bar{P}$ the Zariski closure of $\Ad_{\g}(P)$ in $\GL(\g)$.

We fix once and for all such a point $x_0$. First of all, if $\hx_0\in \pi^{-1}(x_0)$, we have $S.\iota_{\hx_0} \subset P.\iota_{\hx_0}$, so that the conclusions of Proposition \ref{prop:zimmer_conformal} are valid at $x_0$ (its proof is based on this inclusion). We proved at the beginning of Section \ref{s:conformal_flatness} that this implies that the stabilizer at $x_0$ is $\h_{x_0} = \R.Y$, the orbit of $x_0$ is degenerate and at $x_0$ the vector field $X$ gives the direction of the kernel. As it is done in \cite{bader_frances_melnick} and \cite{melnick}, we derive a more elaborated information from the inclusion $\bar{S}.\phi(\hx_0) \subset \bar{P}.\phi(\hx_0)$. Let $\h^{\hx_0} \subset \g$ denote the linear subspace $\iota_{\hx_0}(\h)$.

\begin{prop}
There exists an algebraic subgroup $\bar{P^{\hx_0}} \subset \bar{P} \subset \GL(\g)$ such that $\bar{P^{\hx_0}}$ fixes $\mathcal{D}^{r_0}K(\hx_0)$, $\bar{P^{\hx_0}} \, . \, \h^{\hx_0} \subset \h^{\hx_0}$ and for all $\bar{p} \in \bar{P^{\hx_0}}$, the restriction $\bar{p}|_{\h^{\hx_0}}$ is conjugated by $\iota_{\hx_0}$ to an element of $\bar{S} \subset \GL(\h)$. Moreover, the natural algebraic morphism $\bar{\rho}_{\hx_0} : \bar{P^{\hx_0}} \rightarrow \bar{S}$ is surjective.
\end{prop}

\begin{rem}
The linear action on $\Ad_{\g}(P)$ on the target space $\mathcal{W}^{r_0}$ of the map $\mathcal{D}^{r_0}K$ naturally extends to an action of $\bar{P}$ (more generally to the subgroup of $\GL(\g)$ that preserves $\p$).
\end{rem}

\begin{proof}
Let $\chech{P} := \{\bar{p} \in \bar{P} \ | \ \bar{p} \, . \h^{\hx_0} \subset \h^{\hx_0}\}$ and let $\chech{\rho} : \bar{p} \in \chech{P} \mapsto \iota_{\hx_0}^{-1} \circ \bar{p} |_{\h^{\hx_0}} \circ \iota_{\hx_0} \in \GL(\h)$. The group $\chech{P}$ is algebraic, and so is the morphism $\chech{\rho}$. The inclusion $\bar{S}.\phi(\hx_0) \subset \bar{P}.\phi(\hx_0)$ implies that $\bar{S}$ is contained in the image of $\chech{\rho}$. Define now $\bar{P^{\hx_0}}$ as being the intersection of the preimage of $\bar{S}$ by $\chech{\rho}$ and the stabilizer of $\mathcal{D}^{r_0}K(\hx_0)$ in $\bar{P}$. It is an algebraic group, and if we note $\bar{\rho}_{\hx_0}$ the restriction of $\chech{\rho}$ to this subgroup, $\bar{\rho}_{\hx_0}$ takes values in $\bar{S}$ and is onto.
\end{proof}

The group $\Ad_{\g}(P)$ has finite index in $\bar{P}$. So, $\Ad_{\g}(P)$ contains the identity component of $\bar{P}$, and necessarily the identity component of $\bar{P^{\hx_0}}$. The latter is sent by $\bar{\rho}_{\hx_0}$ onto $\Ad_{\h}(S)$, by connectedness of $S$. We define now $P^{\hx_0} \subset P$ as being the preimage of the identity component of $\bar{P^{\hx_0}}$ by the algebraic morphism $\Ad_{\g} : P \rightarrow \GL(\g)$. Finally, we define the surjective morphism
\begin{equation*}
\rho_{\hx_0} := \bar{\rho}_{\hx_0} \circ \Ad_{\g} : P^{\hx_0} \rightarrow \Ad_{\h}(S).
\end{equation*}

If $\{\e^{tA}\} \subset P^{\hx_0}$ is a one parameter subgroup, then $A \in \Kill^{r_0+1}(\hx_0)$ and Theorem \ref{thm:frobenius} gives a local conformal Killing vector field $A^*$ defined near $x_0$ such that  $\omega_{\hx_0}(\hat{A}^*) = A$. Since $A \in \p$, $\hat{A}^*$ is vertical at $\hx_0$, \textit{i.e.} $A^*(x_0) = 0$ and we have $\phi_{\hat{A}^*}^t(\hx_0) = \hx_0 . \e^{tA}$. This element $A \in \p$ is called the \textit{holonomy} of the conformal vector field $A^*$ at $\hx_0$, in the terminology of \cite{frances_localdynamics}, and it determines the behaviour of $A^*$ near $x_0$. However, relating the dynamic of a conformal vector field near a singularity to the algebraic properties of its holonomy is a difficult issue in general (it is the object of \cite{frances_localdynamics} and \cite{frances_melnick13}). 

\vspace*{.2cm}

Our situation is more comfortable: using the fact that $\bar{\rho}_{\hx_0}$ is algebraic and surjective, we will choose a special element $A \in \p^{\hx_0}$ so that it will be almost immediate that the corresponding local vector field $A^*$ has a linear dynamic, directly related to $\Ad_{\g}(\e^{tA})$.

\subsection{Exhibiting an appropriate holonomy}
\label{ss:holonomy}

Let $\{ p^t \} \subset P^{\hx_0}$ be such that $\rho_{\hx_0}(p^t) = \Ad_{\h}(\e^{tX})$. The one parameter group $\{p^t\}$ admits a Jordan decomposition\footnote{%
The Jordan decomposition is valid in algebraic groups. Here, $P \subset \PO(2,n)$ is not algebraic. Nevertheless, it is the quotient of an algebraic subgroup of $O(2,n)$ by $\{\pm \id\}$. So, when we speak of algebraic properties of elements or subgroups of $P$, we deal with the lift of these elements or subgroups to $O(2,n)$.
}
in $P$, \textit{i.e.} it splits (uniquely) into a commutative product $p^t = p_h^t p_u^t p_e^t$, where $p_h^t \in P$ (resp. $p_u^t \in P$, $p_e^t \in P$) is hyperbolic (resp. unipotent, elliptic)  (see \cite{morris}, §4.3). Because $\Ad_{\g} : P \rightarrow \GL(\g)$ is algebraic, it sends $p_h^t$ (resp. $p_u^t$, $p_e^t$) on an hyperbolic $\bar{p}_h^t$ (resp. unipotent, elliptic) one-parameter subgroup of $\GL(\g)$, whose product is equal to $\Ad_{\g}(p^t)$. Thus, $\Ad_{\g}(p^t) = \Ad_{\g}(p_h^t) \Ad_{\g}(p_u^t) \Ad_{\g}(p_e^t)$ is \textit{the} Jordan decomposition of $\Ad_{\g}(p^t) \in \bar{P^{\hx_0}}$. Therefore, since $\bar{P^{\hx_0}}$ is algebraic, $\Ad_{\g}(p_h^t)$, $\Ad_{\g}(p_u^t)$ and $\Ad_{\g}(p_e^t)$ are in $\bar{P^{\hx_0}}$ (\cite{morris}, 4.3.4), proving that $p_h^t, p_u^t, p_e^t \in P^{\hx_0}$. 

By the same argument, $\rho_{\hx_0}(p^t) = \rho_{\hx_0}(p_h^t) \rho_{\hx_0}(p_u^t) \rho_{\hx_0}(p_e^t)$ is the Jordan decomposition, in $\GL(\h)$, of $\rho_{\hx_0}(p^t) = \Ad_{\h}(\e^{tX})$. The latter is of course an hyperbolic one parameter subgroup. Using once more the uniqueness of the Jordan decomposition, we get $\rho_{\hx_0}(p_u^t) = \rho_{\hx_0}(p_e^t) = \id$. Thus, we have $\rho_{\hx_0}(p_h^t) = \Ad_{\h}(\e^{tX})$: it is not restrictive to assume that $\{p^t\}_{t \in \R}$ is hyperbolic. Therefore, its adjoint action on $\g$ is $\R$-split, meaning that $\{p^t\}$ is in a Cartan subgroup of $G$. We now give a brief description of the Lie algebra $\g \simeq \so(2,n)$.

Let $J = 
\begin{pmatrix}
0 & 0 & 1 \\
0 & I_{n-2} & \\
1 & 0 & 0
\end{pmatrix}
$, and for $u \in \R^n$, let $u^*$ denote $J \! \! ~^{t}\! u$. 

Choose coordinates $x_1,\ldots,x_{n+2}$ on $\R^{2,n}$ in such a way that the quadratic form is written $2x_1x_{n+2} + 2x_2x_{n+1} + x_3^2 + \cdots x_n^2$ and such that the parabolic subgroup $P < PO(2,n)$ is the stabilizer of the isotropic line $[1:0:\cdots:0]$. In such a basis, any matrix of $\so (2,n) $ has the form

\begin{equation}
\label{equ:o2n}
\left (
\begin{array}{c|ccc|c}
\lambda & & u & & 0 \\
\hline
&&&& \\
-v^* & & M & & -u^* \\
&&&& \\
\hline
0 & & v & & -\lambda
\end{array}
\right )
\end{equation}
where $u,v \in \R^n$ and $M \in \so(1,n-1)$, \textit{i.e.} $J \! ~^{t} \! M + MJ = 0$. Moreover the Lie algebra of $P$ corresponds to matrices verifying $v= 0$. We identify a natural Cartan subspace in this presentation:
\begin{equation*}
\text{define } \a = 
\left \{
\begin{pmatrix}
\lambda & 0 & 0 \\
0 & M_{\mu} & 0 \\
0 & 0 & -\lambda
\end{pmatrix}
, \ \lambda \in \R, \ \mu \in \R
\right \}
\text{ where }
M_{\mu} :=
\begin{pmatrix}
\mu & 0 & 0 \\
0   & 0 & 0 \\
0 & 0 & -\mu
\end{pmatrix}
\in \so(1,n-1),
\end{equation*}
(the zero in the center of $M_{\mu}$ having size $(n-2)\times(n-2)$). On can verify that the $M_{\mu}$'s form a Cartan subspace of $\so(1,n-1)$ and that $\a$ is a Cartan subspace of $\g$. Thus $p^t$ is conjugated in $G$ to a one-parameter subgroup of $\exp(\a)$. We prove that in fact, the conjugacy can be done \textit{in $P$}.

\begin{lem}
Let $A \in \p$ be an $\R$-split matrix. There exists $p \in P$ such that $\Ad(p)A \in \a$.
\end{lem}

\begin{proof}
We know that $A$ is an upper-triangular matrix of the form
\begin{equation*}
A = 
\begin{pmatrix}
\lambda & u & 0\\
 & M & -u^* \\
 &  & -\lambda
\end{pmatrix}
\end{equation*}
with $\lambda \in \R$ and $M \in \so(1,n-1)$. Since $A$ is diagonalizable over $\R$, so is $M$. Since $\Ad_{\so(1,n-1)} : \SO(1,n-1) \rightarrow \GL(\so(1,n-1))$ is algebraic, $\Ad_{\so(1,n-1)}(\e^{tM})$ is $\R$-split, proving that $M$ is in a Cartan subspace of $\so(1,n-1)$. Thus, we have $g \in \SO(1,n-1)$ and $\mu \in \R$ such that $M = \Ad(g) M_{\mu}$ and
\begin{equation*}
\Ad(p_0)A = : A' =
\begin{pmatrix}
\lambda & u' & 0\\
 & M_{\mu} & -u'^* \\
 &  & -\lambda
\end{pmatrix},
\text{ where }
p_0 := 
\begin{pmatrix}
1 & & \\
 & g & \\
 & & 1
\end{pmatrix} 
\in P \text{ and } u' = u.g^{-1}.
\end{equation*}
We note 
\begin{equation*}
A_{\ell} = 
\begin{pmatrix}
\lambda & u' & 0\\
 & M_{\mu} & -u'^* \\
 &   & -\lambda
\end{pmatrix}
, \text{ and for all } v \in \R^n, \ T(v) =
\begin{pmatrix}
0 & v & 0 \\
  & 0 & -v^* \\
  & & 0
\end{pmatrix},
\end{equation*}
so that $A' = A_{\ell} + T(u')$. If $v \in \R^n$, we have $[A_{\ell},T(v)] = ((\lambda - \mu)v_1,\lambda v_2,\ldots,\lambda v_{n-1},(\lambda+\mu)v_n)$. We set $\lambda_1 = \lambda - \mu$, $\lambda_2 = \cdots = \lambda_{n-1} = \lambda$ and $\lambda_n = \lambda+\mu$ and define $v^0 \in \R^n$ by
\begin{equation*}
\forall 1 \leq i \leq n, \ v_i^0 = 
\begin{cases}
u_i' / \lambda_i \text{ if } \lambda_i \neq 0 \\
0 \text{ else}.
\end{cases}
\end{equation*}
Finally, if $p= \exp(T(v^0)) \in P$, we have $A'' := \Ad(p)A' = A' + [T(v^0),A'] = A_{\ell} + T(u'')$, with $u'' \in \R^n$ such that $[A_{\ell},T(u'')]=0$. Thus, $A''$ and $A_{\ell}$ commute and are $\R$-split. So, $T(u'')$ is also $\R$-split, proving $T(u'') = 0$ and $A = \Ad(p_0^{-1}p^{-1}) A_{\ell}$.
\end{proof}

Thus, we know that our one-parameter subgroup $\{p^t\} \subset P^{\hx_0}$ is conjugated in $P$ to a one-parameter subgroup of $\exp(\a)$. Note that the point $\hx_0$ in the fiber of $x_0$ has been chosen arbitrarily. We compute that $P^{\hx_0.p} = pP^{\hx_0}p^{-1}$ for every $p \in P$. Finally, there is a choice of $\hx_0 \in \pi^{-1}(x_0)$ such that there exists $A \in \a$ verifying $\e^{tA} \in P^{\hx_0}$ and $\rho_{\hx_0}(\e^{tA}) = \Ad_{\h}(\e^{tX})$. Until the end of this section, $\hx_0$ denotes this particular element in the fiber $\pi^{-1}(x_0)$.

\subsection{Linear dynamic of $\phi_{A^*}^t$}

By definition of $\rho_{\hx_0}$, $\Ad_{\g}(\e^{tA}) \iota_{\hx_0}(X) = \iota_{\hx_0}(\Ad_{\h}(\e^{tX}).X) = \iota_{\hx_0}(X)$. Since $X_{x_0}$ is an isotropic tangent vector in $T_{x_0}M$, the projection of $\iota_{\hx_0}(X)$ in $\g/\p$ is isotropic with respect to $Q$ (cf. Section \ref{ss:cartan_geometry}). If we take back the description of $\so(2,n)$ in (\ref{equ:o2n}), identifying $\g / \p \simeq \n_-$ where $\n_- \subset \g$ denotes the space of strictly lower triangular matrices
\begin{equation*}
\n_- =
\left \{
\begin{pmatrix}
& & \\
-v^* & \\
0 & v &
\end{pmatrix}
, \ v=(v_1,\ldots,v_n) \in \R^n
\right \},
\end{equation*}
the quadratic form $Q$ is a positive multiple of $2v_1v_n + v_2^2 + \cdots + v_{n-1}^2$. Note that $\Ad_{\g}(\e^{tA})$ preserves the subspace $\n_-$, so that its action on $\g/\p$ is conjugated to its restriction to $\n_-$. Precisely, there are $\lambda$ and $\mu$ such that the action of $\Ad_{\g}(\e^{tA})$ on $\n_-$ is $(\e^{(\mu - \lambda)t} v_1, \e^{-\lambda t}v_2,\ldots, \e^{-\lambda t}v_{n-1},\e^{-(\mu +\lambda)t}v_n)$. Since $\bar{\Ad_{\g}}(\e^{tA})$ fixes an isotropic vector in $\g / \p$, we see that up to exchanging $v_1$ and $v_n$ and rescaling $A$, in the coordinates $(v_1,\ldots,v_n)$,
\begin{equation}
\label{equ:allure_flot}
\Ad_{\g}(\e^{tA})|_{\n_-} =
\begin{pmatrix}
1 & & & & \\
  & \e^{-t} & & & \\
  & & \ddots & & \\
  & & & \e^{-t} & \\
  & & & & \e^{-2t}
\end{pmatrix}.
\end{equation}

The final ingredient is the notion of \textit{conformal exponential chart}. If $X_0 \in \g$, we note $\tilde{X_0} := \omega^{-1}(X_0)$ the $\omega$-constant vector field on $\hat{M}$ associated to $X_0$. If $X_0 \in \g$ is small enough, set $\exp(\hx,X_0) = \phi_{\tilde{X_0}}^1(\hx)$ (the local flow of $\tilde{X_0}$ at time $1$) for $\hx \in \hat{M}$. We get a local diffeomorphism $\exp(\hx, .) : \mathcal{U} \subset \g \rightarrow \hat{U} \subset \hat{M}$ where $\mathcal{U}$ is a neighbourhood of the origin and $\hat{U}$ is neighbourhood of $\hx$.

\vspace*{0.2cm}

Consider $A^*$ the local conformal vector field given by Theorem \ref{thm:frobenius}. Since $\omega_{\hx_0}(\hat{A}^*) = A$, the lift $\hat{A}^*$ is tangent to the fiber $\pi^{-1}(x_0)$, so that its local flow preserves this fiber and satisfies
\begin{equation*}
\hat{\phi_{A^*}^t}(\hx_0) = \hx_0 . \e^{t A}.
\end{equation*}
Since $A^*$ is conformal, its lift commutes with $\omega$-constant vector fields. Therefore, for small $X_0 \in \g$ and $t$, we have
\begin{align*}
\hat{\phi_{A^*}^t}(\exp(\hx_0,X_0)) & = \exp(\hat{\phi_{A^*}^t}(\hx_0),X_0) = \exp(\hx_0 . \e^{tA},X_0) \\
                                   & = \exp(\hx_0,\Ad_{\g}(\e^{tA})X_0).\e^{tA},
\end{align*}
the last equality coming from the property $\forall p \in P$, $(R_p)_* \tilde{X_0} = \tilde{\Ad_{\g}(p^{-1})X_0}$ where $R_p$ denotes the right action of $p$ on the Cartan bundle $\hat{M}$ (third property of the Cartan connection).

\vspace*{0.2cm}

Define $\psi : \mathcal{U} \cap \n_- \rightarrow U:=\pi(\hat{U})$ by $\psi(X_0) = \pi(\exp(\hx_0,X_0))$. If $\mathcal{U} \subset \g = \n_- \oplus \p$ has been chosen of the form $\mathcal{U}_{\n_-} \times \mathcal{U}_{\p}$, with $\mathcal{U}_{\n_-}$ a cylinder of the form $]-\epsilon,\epsilon[ \times B$ in the coordinates $(v_1,\ldots,v_n)$, with $B$ an open Euclidian ball in $\{v_1 = 0\}$, then $\mathcal{U}_{\n_-}$ is preserved by $\Ad_{\g}(\e^{tA})$ for all positive times $t$. Moreover, we see that if $\phi_{A^*}^t$ is defined at time $t$, then for all $X_0 \in \mathcal{U}_{\n_-}$
\begin{equation*}
\phi_{A^*}^t(\psi(X_0)) = \pi (\hat{\phi_{A^*}^t}(\exp(\hx_0,X_0))) = \pi(\exp(\hx_0,\Ad_{\g}(\e^{tA})X_0)) = \psi(\Ad_{\g}(\e^{tA})X_0).
\end{equation*}

This shows that the local flow of $A^*$ is conjugated by $\psi$ to $\Ad_{\g}(\e^{tA})|_{\n_-}$. In particular, it is defined for all positive times on $\psi(\mathcal{U}_{\n_-})$ and is conjugated by $\psi$ to the linear flow given in (\ref{equ:allure_flot}). At last, if we identify $\n_- \simeq T_{x_0}M$ via $X_0 \mapsto \pi_* (\tilde{X_0})_{\hx_0}$, the restriction $\Ad_{\g}(\e^{tA})|_{\n_-}$ is conjugated to the isotropy $T_{x_0}\phi_{A^*}^t$ and $X_{x_0}$ corresponds to $\frac{\partial}{\partial v_1}(x_0)$ (the letter $X$ refers to the $\sl(2)$-triple $(X,Y,Z)$ we fixed at the beginning of the proof).

\vspace*{0.2cm}

We finally prove $[A^*,X] = 0$ on $U$. To see that this vector field vanishes identically on $U$, it is enough to prove that its lift, which is $[\hat{A}^*,\hat{X}]$, vanishes at some point in $\hat{U}$ by connectedness of $U$. Since $\rho_{\hx_0}(\e^{tA}) = \Ad_{\h}(\e^{tX})$, we have $[A , \iota_{\hx_0}(X)] = 0$, \textit{i.e.} $[\omega_{\hx_0}(\hat{A}_{\hx_0}),\omega_{\hx_0}(\hat{X}_{\hx_0})] = 0$. A standard property of Cartan connections (see \cite{sharpe}, p.192) ensures
\begin{equation*}
\omega_{\hx_0}([\hat{A}^*,\hat{X}]_{\hx_0}) = [\omega_{\hx_0}(\hat{A}_{\hx_0}^*),\omega_{\hx_0}(\hat{X}_{\hx_0})] + \Omega_{\hx_0}(\hat{A}_{\hx_0}^*,\hat{X}_{\hx_0}).
\end{equation*}
where $\Omega := \d \omega + \frac{1}{2}[\omega,\omega]$ denotes the curvature form. Another general property of this $2$-form is its horizontality: it vanishes as soon as one of its arguments is vertical (\cite{sharpe}, Ch.5, Cor.3.10). This finishes the proof since $\hat{A}^*$ is tangent to the fiber $\pi^{-1}(x_0)$.

\bibliographystyle{amsalpha}
\bibliography{references_sl2.bib}
\nocite{*}

\vspace*{1cm}

\hspace*{8cm}
\begin{minipage}[b]{0.55\linewidth}
Vincent \textsc{Pecastaing} \\
Institut für Mathematik \\
Humboldt University of Berlin \\
\texttt{vincent.pecastaing@normalesup.org}
\end{minipage}
\end{document}